\documentclass[11pt]{article} % use larger type; default would be 10pt

\usepackage[utf8]{inputenc} % set input encoding (not needed with XeLaTeX)

\usepackage{geometry} % to change the page dimensions
\usepackage{graphicx} % support the \includegraphics command and options
\usepackage[font=sf]{caption,subcaption}
\usepackage{hyperref}
\hypersetup{
    colorlinks=false,
    linkcolor=blue,
    filecolor=blue,      
    urlcolor=blue,
}
\usepackage{array} % for better arrays (eg matrices) in maths 
\usepackage{paralist} % very flexible & customisable lists (eg. enumerate/itemize, etc.)
\usepackage{verbatim} % adds environment for commenting out blocks of text & for better verbatim
\usepackage{lipsum}
\usepackage{amsmath, amsfonts, amssymb, amsxtra, amsthm, mathrsfs}
\usepackage{enumitem}
\usepackage{color}
\usepackage{tikz,tikz-3dplot,pgfplots}
\usepackage{multirow}
\usepackage{hyperref}
\hypersetup{
 colorlinks = true,
     linkcolor = blue,
     anchorcolor = blue,
     citecolor = brown,
     filecolor = blue,
     urlcolor = blue
}
\usepackage[titles]{tocloft}
\usepackage{authblk}

\theoremstyle{plain}
\newtheorem{thm}{Theorem}[section]
\newtheorem{prop}[thm]{Proposition}
\newtheorem{lm}[thm]{Lemma}
\newtheorem{cor}[thm]{Corollary}

\newtheorem{ques}[thm]{Question}
\newtheorem{clm}{Claim}
\newtheorem{clmA}{Claim}

\theoremstyle{definition}

\newtheorem{rmk}[thm]{Remark}

\newtheorem{alg}[thm]{Algorithm}

\DeclareMathOperator{\impart}{Im}

\DeclareMathOperator{\hull}{hull}

\DeclareMathOperator{\Int}{Int}

\DeclareMathOperator{\bolda}{\mathbf{a}}
\DeclareMathOperator{\boldb}{\mathbf{b}}
\definecolor{cof}{RGB}{219,144,71}
\definecolor{pur}{RGB}{186,146,162}
\definecolor{greeo}{RGB}{91,173,69}
\definecolor{greet}{RGB}{52,111,72}
\usetikzlibrary{calc}

\pgfplotsset{compat=1.14}

\usepackage[calcwidth]{titlesec}
\titlespacing*{\paragraph}{0pt}{0pt}{1em}

\numberwithin{equation}{section}

\title{\bfseries Train tracks, entropy, and the halo of a measured lamination}
\author{Tina Torkaman}
\author{Yongquan Zhang}
\affil{\small Department of Mathematics, Harvard University}

\begin{document}

\maketitle

\begin{abstract}
Let $\mathcal{L}$ be a measured geodesic lamination on a complete hyperbolic surface of finite area. Assuming $\mathcal{L}$ is not a multicurve, our main result establishes the existence of a geodesic ray which has finite intersection number with $\mathcal{L}$ but is not asymptotic to any leaf of $\mathcal{L}$ nor eventually disjoint from $\mathcal{L}$. In fact, we show that the endpoints of such rays, when lifted to the universal cover $\mathbb{H}^2$ of $X$, give an uncountable set $h\tilde{\mathcal{L}}\subset S^1$ (called the \emph{halo} of $\tilde{\mathcal{L}}$), which is disjoint from the endpoints of leaves of the lifted lamination $\tilde{\mathcal{L}}$.

\end{abstract}
\begin{center}
\begin{minipage}{0.9\linewidth}
    \tableofcontents
\end{minipage}
\end{center}
\thispagestyle{empty}
\newpage

\clearpage
\pagenumbering{arabic} 
\section{Introduction}

In this paper we investigate the intersection of geodesic rays with a given measured geodesic lamination $\mathcal{L}$ on a hyperbolic surface $X$ of finite area. Our main result establishes the existence of exotic rays; these rays have the unexpected property that they have finite intersection number with $\mathcal{L}$, even though they meet $\mathcal{L}$ infinitely often. We are motivated by applications to bending laminations and hyperbolic 3-manifolds, described below.
%which is a geodesic ray with finite intersection number with the lamination, but neither asymptotic to any leaf nor disjoint from the lamination.We are motivated by applications to bending laminations and hyperbolic 3-manifolds, which we will describe later.

\paragraph{Exotic rays.}
Let $X:=\Gamma\backslash\mathbb{H}^2$ be a complete hyperbolic surface of finite area, and $\mathcal{L}$ a measured geodesic lamination on $X$. By a geodesic ray on $X$ we mean a geodesic isometric immersion $r:[0,\infty)\to X$.

Define the \emph{intersection number} $I({\mathcal{L}},r)$ as the transverse measure of $r$ with respect to $\mathcal{L}$. For almost every ray $r$, $I({\mathcal{L}},r)$ is infinite (see Theorem~\ref{thm:linear}). On the other hand, $I({\mathcal{L}},r)$ is finite when:

%We say two geodesic rays $r_1$ and $r_2$ are \emph{asymptotic} if there exists a constant $t_0$ so that the hyperbolic distance $d(r_1(t),r_2(t+t_0))\to 0$ as $t\to\infty$. Equivalently, some lifts of the two geodesic rays to the universal cover $\tilde X\cong\mathbb{H}^2$ share the same endpoint at infinity. 

%The main focus of the paper is the following question: does there exist a geodesic ray $r$ so that intersection number of $r$ with $\mathcal{L}$ is finite? We have the following two natural cases:
\begin{itemize}[topsep=0mm, itemsep=0mm]
    \item $r$ is asymptopic to a leaf of $\mathcal{L}$ (see Proposition~\ref{prop:point_at_infinity}); or
    \item $r$ is eventually disjoint from $\mathcal{L}$.
\end{itemize}
We call a ray \emph{exotic for $\mathcal{L}$} if $I(\mathcal{L},r)$ is finite but it belongs to neither of these cases. Our main question is then the following: given a measured geodesic lamination $\mathcal{L}$, does there exist an exotic ray for $\mathcal{L}?$

If $\mathcal{L}$ is a multicurve, i.e.~
$
\mathcal{L}=a_1\gamma_1+\dots+a_k\gamma_k
$
where $a_i \in \mathbb{R}$ and $\gamma_i$'s are simple closed geodesics, then there is no exotic ray; any ray with $I(r,\mathcal{L})<\infty$ must eventually be disjoint from $\mathcal{L}$. Our main result addresses the remaining cases: 

%Our main result concerns geodesic rays that have finite intersection with $\mathcal{L}$. . If $\mathcal{L}$ contains a minimal component which is not a multicurve, any geodesic ray asymptotic to a leaf of that component has finite intersection number with $\mathcal{L}$ (see Proposition~\ref{prop:point_at_infinity}). A geodesic ray is said to be \emph{exotic} if it has finite intersection with $\mathcal{L}$, but is not asymptotic to any leaf of $\mathcal{L}$ nor eventually disjoint from $\mathcal{L}$. We have
%, i.e. a geodesic ray $r$ with $\mu_r(\infty):=\lim_{t\to\infty}\mu_r(t)<\infty$, or equivalently, $\mu_r\asymp 0$ or $1$. If $r$ is eventually disjoint from $\mathcal{L}$, $\mu_r(\infty)$ is obviously finite. Its total transverse measure is also finite if $r$ is asymptotic to a leaf $l$ of $\mathcal{L}$ (see Prop.~\ref{prop:point_at_infinity}). A geodesic ray $r$ is called \emph{exotic} (for $\mathcal{L}$) if it has finite total transverse measure and is neither eventually disjoint from $\mathcal{L}$ nor asymptotic to a leaf of $\mathcal{L}$. We have
\begin{thm} \label{main1}
Provided that $\mathcal{L}$ is not a multicurve, there exists an exotic ray for $\mathcal{L}$.
\end{thm}
%In contrast, if a geodesic ray $r$ has finite transverse measure with respect to a purely atomic measured lamination $\mathcal{L}$, it is eventually disjoint from $\mathcal{L}$, as each intersection increases the measure by a definite amount. We remark that exotic rays in our construction still ``tend to" the lamination, in the sense that points on the ray are getting closer and closer to $\mathcal{L}$; see e.g. Lemma~\ref{lm:close}.
\paragraph{The halo of a measured lamination.}
 We now put Theorem~\ref{main1} into a broader context. Let $\mathcal{M}$ is a measured geodesic lamination on $\mathbb{H}^2$. Let $\partial\mathcal{M}\subset S^1$ be the set of end points of geodesics in $\mathcal{M}$. We can define \emph{exotic rays} of $\mathcal{M}$ as above. The \emph{halo} of $\mathcal{M}$, denoted by $h\mathcal{M}$, is the set of end points of exotic rays for $\mathcal{M}$. By definition, $h\mathcal{M}\cap \partial\mathcal{M}=\varnothing$.

We then have the following stronger version of Theorem~\ref{main1}:
\begin{thm}\label{mainprime}
Let $\mathcal{L}$ be a measured geodesic lamination on a complete hyperbolic surface $X=\Gamma\backslash\mathbb{H}^2$ of finite area, and $\tilde{\mathcal{L}}$ its lift to $\mathbb{H}^2$. Then the halo $h\widetilde{\mathcal{L}}$ is either empty or uncountable. Moreover, it is uncountable if and only if $\mathcal{L}$ is not a multicurve.
\end{thm}

\paragraph{Scheme for constructing an exotic ray.}
We now give a brief description of the construction in the proof of the main result. The key challenge is to find a way of determining whether $r$ is asymptotic to a leaf of $\mathcal{L}$, while keeping track of $I(r,\mathcal{L})$. For this, we proceed as follows:
\begin{enumerate}[topsep=0mm, itemsep=0mm]
    \item First, we assign an infinite symbolic word to each geodesic ray. This is done using train tracks carrying $\mathcal{L}$ (see \S\ref{sec:entropy}), then two rays are asymptotic if and only if their corresponding words have the same tail (see Theorem~\ref{thm:tail2}).
    \item Next, we show that there are many \emph{inadmissible words}, which are finite words not contained in the words of any leaves. This follows from the fact that entropy of admissible words is zero, although the entropy of all words is positive (Proposition~\ref{prop:ent.poly}). 
    
    %the words for leaves of $\tilde{\mathcal{L}}$ make up a small portion of all infinite words by comparing their entropy (Proposition~\ref{prop:ent.poly}); 
    %in fact, the words for leaves avoid many finite segments which we call \emph{inadmissible words} (\S\ref{sec:general_proof}, proof of Theorem~\ref{mainprime}).
    \item We then show that there exist geodesic segments represented by inadmissible words with arbitrarily small transverse measure. By concatenating inadmissible words of smaller and smaller transverse measure we obtain a piecewise geodesic ray with finite transverse measure.
    %In particular, given any sequence of positive numbers $(\epsilon_k)$, there exist inadmissible words $(w_k)$ so that a representative arc of $w_k$ has transverse measure $<\epsilon_k$ (\S\ref{sec:general_proof}, proof of Theorem~\ref{mainprime}).
    \item Finally, we show that the geodesic representative of the constructed ray also has finite transverse measure (Proposition~\ref{prop:point_at_infinity}), and is indeed exotic. 
    %Finally, by concatenating inadmissible words (and connecting words of negligible transverse measure if necessary) $w_1\cdots w_2\cdots w_3\cdots$, we obtain an exotic ray by insuring that $\sum\epsilon_k<\infty$. We call the infinite word obtained in this step an \emph{exotic word}.
\end{enumerate}

\paragraph{The case of a punctured torus.}
We now describe the special case where $X$ is a punctured torus. For this case, we do not need to explicitly use train tracks or entropy in our construction.

Let $X:=\Gamma\backslash\mathbb{H}^2$ be a complete hyperbolic torus with one puncture. A measured lamination on $X$ is determined up to scale by its \emph{slope} $\theta\in\mathbb{P}H^1(X,\mathbb{R})\cong\mathbb{RP}^1$. Multicurves are identified with the rational points $\mathbb{QP}^1$.

Choose an ideal quadrilateral $Q$ in $\mathbb{H}^2$ as a fundamental domain for $\Gamma$, so that $\bolda, \boldb$ are two hyperbolic elements identifying the opposite sides of $Q$. Then $\Gamma$ is the free group $\langle \bolda,\boldb\rangle$. Note that the orbit of $Q$ under $\Gamma$ tiles $\mathbb{H}^2$. Label each tile $gQ$ by $g\in\Gamma$; see Figure~\ref{fig:symbolic_coding}. For any geodesic ray $r$ in $\mathbb{H}^2$ based in $Q$ and ending at a point $z$ that is not a cusp, there is a unique infinite reduced word $g=g_1g_2g_3\cdots$ in the generators $\bolda,\boldb,\overline{\bolda}=\bolda^{-1},\overline{\boldb}=\boldb^{-1}$ so that $g_1g_2\cdots g_nQ\to z$. In fact, $Q_n=g_1\cdots g_nQ$ is the sequence of tiles that $r$ passes through.

\begin{figure}[ht!]
    \centering
    \includegraphics[trim={0cm 0.5cm 0cm 0.5cm}, clip, scale=0.75]{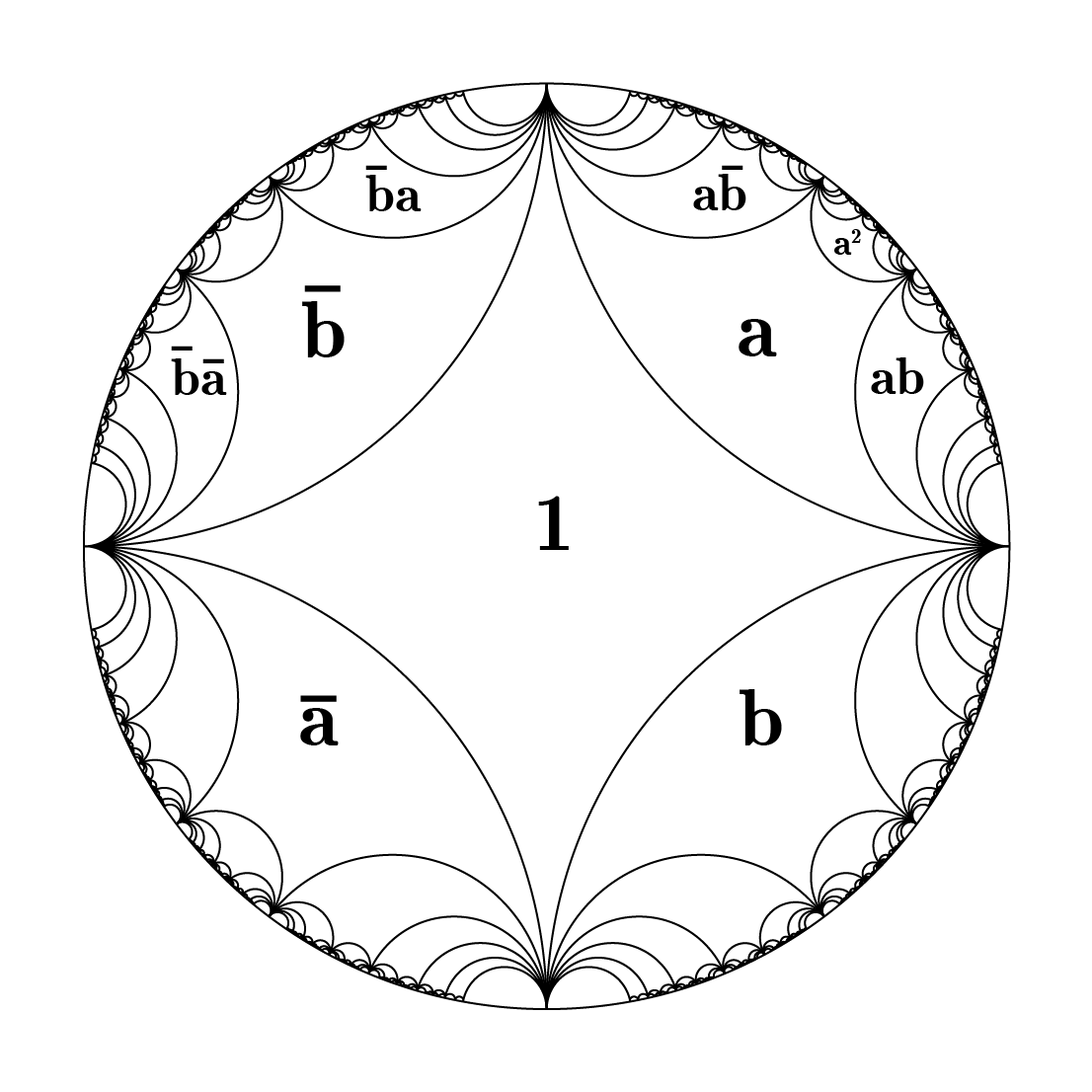}
    \caption{Symbolic coding with the quadrilateral $Q$ labelled $1$}
    \label{fig:symbolic_coding}
\end{figure}

Given a geodesic ray on $X$, there is a unique lift to $\mathbb{H}^2$ with the base point in $Q$ (if it is based at an edge of $Q$, we always choose the edge to be the one shared with $\bolda Q$ or $\boldb Q$). It is easy to see that two geodesic rays are asymptotic on $X$ if and only if the words obtained above for their lifts have the same tail.

Bi-infinite words coming from leaves of a measured lamination of $X$, called \emph{Sturmian words}, are quite well-studied (see e.g. \cite{sturmian1}, \cite{sturmian2}, and for an exposition, \cite{sturmian}), and are closely related to the slope $\theta$ of the measured lamination, via the continued fraction of $\theta$ and its best rational approximations. A finite word is said to be ($\theta$-)\emph{inadmissible} if it is not contained in any Sturmian word of slope $\theta$. In \S\ref{sec:tori}, we give a self-contained account of some properties of Sturmian words. In particular, we prove
\begin{prop}\label{prop:tori}
Let $p_k/q_k$ be the $k$-th continued fraction approximation of an irrational number $\theta$. Then there exists a $\theta$-inadmissible word $w_k$ of length $\le2(p_k+q_k)$ and transverse measure $\leq c/q_k$, where $c$ is a constant depending only on $\theta$.
%\item For any increasing sequence $(i_k)$ of natural numbers, the concatenation $w_{i_1}w_{i_2}w_{i_3}\cdots$ gives an exotic ray.
\end{prop}
The proof is constructive; the algorithm that produces such words is described in \S\ref{sec:tori}. Since $\sum 1/q_k<\infty$, we immediately have the the following special case of Theorem~\ref{main1}:
\begin{thm}\label{thm:tori}
For any increasing sequence $(i_k)$ of integers of the same parity, the concatenation $w_{i_1}w_{i_2}w_{i_3}\cdots$ gives an exotic word.
\end{thm}
The restriction to integers of the same parity is to guarantee that connecting the segments does not create too much transverse measure; see \S\ref{sec:tori} for details.

%We remark that while we prove Proposition~\ref{prop:tori} and Theorem~\ref{thm:tori} following the scheme described above, the symbolic coding is based on flat structure instead of train tracks.
\paragraph{Jordan domain, 3-manifolds, and exotic circles.}
To conclude the introduction, we will discuss the connection of exotic rays with Jordan curves and quasifuchsian groups, which has motivated our study.

One source of measured laminations on $\mathbb{H}^2$ is the bending lamination of the convex hull of a Jordan curve. Let $\Omega$ be a Jordan domain in $\hat{\mathbb{C}}=\mathbb{C}\cup\{\infty\}$, and let $\Lambda$ be its boundary.

We may view $\hat{\mathbb{C}}\cong S^2$ as the boundary at infinity of $\mathbb{H}^3$. Let $\hull(\Lambda)$ be the convex hull of $\Lambda$ in $\mathbb{H}^3$. Unless $\Lambda$ is a round circle, $\partial\hull(\Lambda)$ consists of two connected components, each isometric to $\mathbb{H}^2$ in the metric inherited from $\mathbb{H}^3$ and bent along a bending lamination (see e.g.\ \cite{bending}). Let $H$ be the component corresponding to $\Omega$ with bending lamination $\mathcal{L}$.

An \emph{exotic circle} in $\Omega$ is a circle $C\subset\overline{\Omega}$ so that $C\cap\Lambda$ consists of a single point $p$, and any larger circle enclosing $C$ is not entirely contained in $\overline{\Omega}$. See Figure~\ref{fig:ellipse} for an example.
\begin{figure}[ht!]
    \centering
    \captionsetup{width=.8\linewidth}
    \includegraphics{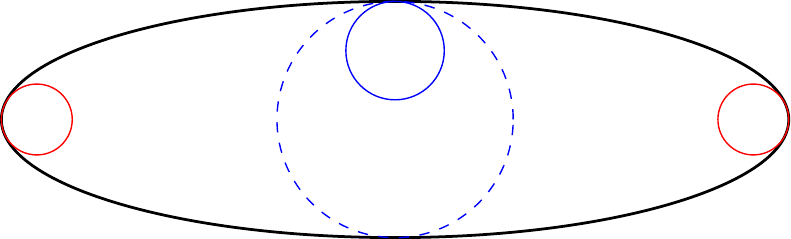}
    \iffalse
    \begin{tikzpicture}
    \draw[thick] (0,0) ellipse (3 and 1.5);
    \draw[red] (2.25,0) circle (0.75);
    \draw[red] (-2.25,0) circle (0.75);
    \draw[blue] (0,1) circle (0.5);
    \draw[blue, dashed] (0,0) circle (1.5);
    \end{tikzpicture}
    \fi
    \caption{A domain bounded by an ellipse, its exotic circles (in red), and a non-exotic circle (in blue). The corresponding bending lamination $\mathcal{M}$ gives a foliation of $\mathbb{H}^2$ by geodesics, and $h\mathcal{M}$ consists of two points.}
    \label{fig:ellipse}
\end{figure}

Using Gauss-Bonnet theorem, it is easy to show
\begin{prop}\label{prop:exotic_circle_necessary}
If $C$ is an exotic circle in $\Omega$, then the point $p=C\cap\Lambda$ is contained in the halo of the bending lamination $\mathcal{L}$.
\end{prop}
In particular, this applies when $\Lambda$ is the limit set of a quasifuchsian manifold $M$. We have the following question:
\begin{ques}
Does the domain of discontinuity of a typical quasifuchsian manifold exhibit an exotic circle?
\end{ques}

By Proposition~\ref{prop:exotic_circle_necessary} and our main theorem, if the bending lamination on one end is a multicurve, exotic circles do not exist on that side. On the other hand, the bending lamination of a ``typical" quasifuchsian manifold is not a multicurve on at least one end. However, the results in this paper is not enough to establish the existence of an exotic circle.

For a quasifuchsian manifold $M$, exotic circles for its limit set $\Lambda$ correspond to \emph{exotic planes} contained in an end of $M$. Such a plane accumulates on the convex core boundary of $M$, but there is no support plane separating it from the boundary. Determining whether an exotic plane could exist at all, and if so, giving a sufficient condition when it exists are important for analyzing the topological behavior of geodesic planes contained in an end of the quasifuchsian manifold. This is the main motivation for the present work and will be addressed in future research.

%\paragraph{Some complementary results.}

%Among these results, Theorem~\ref{main1} is probably the most interesting. The study of the existence of an exotic ray is motivated by the second author's investigation on geodesic planes in quasifuchsian manifolds (see the discussion below). Theorem~\ref{main1} and its proof, however, are presented without any reference to 3-manifolds, while involving many aspects of surface topology, geometry and combinatorics. The remainder of the introduction is devoted to outlining the basic ideas of the proof.

%To actually implement the scheme, we also need to address some technical difficulties: symbolic coding only gives piecewise geodesic rays, and straightening them to geodesic rays could affect the coding; or the tails of the words only determine asymptotic behavior of a subset of all geodesic rays. See \S\ref{sec:general_proof} for details.

%We also give an argument for uncountability here. Let $w_k$ and $\epsilon_k$ be as in the scheme, and assume $\sum\epsilon_k<\infty$. Given any increasing sequence $(i_k)$ of natural numbers, $w_{i_1}\cdots w_{i_2}\cdots w_{i_3}\cdots$ gives an exotic word (here $\cdots$ denotes connecting words of negligible transverse measure). For the corresponding words to share a tail, the sequences must share a tail as well. In particular, there are at most countably many sequences that give words sharing the same tail. This gives uncountably many exotic words not sharing tails, as desired.
\paragraph{Notes and references.}
The theory of Sturmian words has been generalized to regular octagon and even all regular $2n$-gons, see \cite{octagon}. Our discussions in the case of punctured tori is also inspired and guided by \cite{simple}, where properties of simple words of a punctured hyperbolic surface are studied.

We would like to thank C. McMullen for his continuous support, enlightening discussions, and suggestions. Also, Figures~\ref{fig:symbolic_coding} is produced using his program \textbf{lim}.%Figures~\ref{tr1}, \ref{bgn1} and \ref{bgn2} are produced with the open source vector graphics software \textbf{Inkscape}.

\section{Background}\label{sec:background}
In this section, we explain some concepts and related results which will be used later in the proofs.

\paragraph{Measured laminations.}
 A \emph{measured geodesic lamination} $\mathcal{L}$ is a compact subset of $X$ foliated by simple geodesics, together with a transverse invariant measure, which assigns a measure for any arc transverse to the lamination. The total mass of a transverse arc $\Lambda$ with respect to this measure is the \emph{intersection number} of the arc with $\mathcal{L}$, which we denote by $I(\mathcal{L},\Lambda)$.

A measured lamination $\mathcal{L}$ is called \emph{minimal} if every leaf is dense in $\mathcal{L}$. In general, $\mathcal{L}$ consists of finitely many minimal components. For a multicurve, every minimal component is a simple closed geodesic; the measure of any transverse arc is then simply a weighted count of intersections with $\mathcal{L}$.

\paragraph{Geodesic currents.}
For the proof of Theorem~\ref{thm:linear}, we need some basic facts about geodesic currents. We refer to \cite{Bon.gc}, \cite{Bon.Tch} for details.

Given a hyperbolic surface $X$, a \emph{geodesic current} is a locally finite measure on $T_1(X)$, invariant under the geodesic flow $\phi_t$ and an involution $\iota$ defined as follows. For a point $x\in X$ and a unit tangent vector $v$ at $x$, $\iota(x,v)=(x,-v)$. As an example, the \emph{Liouville measure} $\lambda$ is a geodesic current, which locally decomposes as the product of the area measure on $X$ and the uniform measure of total mass $\pi/2$ in the bundle direction\footnote{We adopt the same normalization of $\lambda$ as \cite{Bon.gc}, so that $i(C,\lambda)=\ell(C)$ for any geodesic current $C$.}. Closed geodesics also give examples of geodesic currents as follows. For a closed geodesic $\gamma$ of length $T$, let $r_\pm:[0,T]\to X$ be two arclength parametrizations with opposite orientations. The geodesic current associated to $\gamma$ is then $((r'_+)_*dt+(r'_-)_*dt)/2$, the average of the pushforward of the Lebesgue measure $dt$ on $[0,T]$ under $r'_\pm:[0,T]\to T_1(X)$.

Let $\mathcal{C}(X)$ be the set of all geodesic currents on $X$. We can define notions of length and intersection number for geodesic currents, which extend the usual length and intersection number for closed geodesics. Indeed, the \emph{length} of a geodesic current $C$ is defined to be the total $C-$mass of $T_1(X)$ and is denoted by $\ell(C)$. The \emph{intersection number} of two geodesic currents $C_1$ and $C_2$, denoted by $i(C_1,C_2)$, is harder to define succinctly; for the precise definition see \cite{Bon.gc}. Here are some properties of $i(\cdot,\cdot)$ needed in the proof of Theorem~\ref{thm:generic}:
\begin{itemize}[topsep=0mm, itemsep=0mm]
    \item For any closed geodesics $\gamma_1$ and $\gamma_2$, $i(\gamma_1,\gamma_2)$ gives the number of times they intersect on $X$ with multiplicity.
   % \item $i:\mathcal{C}(X) \times \mathcal{C}(X) \to \mathbb{R}$ is a continuous function. 
    \item For any geodesic current $C$, $i(C,\lambda)=\ell(C)$.
    \item $i(\lambda,\lambda)=\ell(\lambda)=\pi^2(2g-2+n)$, and so $\lambda/(\pi^2(2g-2+n))$ is a probability measure on $T_1(X)$.
\end{itemize}

One fundamental fact concerning the intersection number is continuity. Let $K$ be a compact subset of $X$. We denote by $\mathcal{C}_K(X)$ the collection of geodesic currents whose support in $T_1(X)$ projects down to a set contained in $K$. Then we have \cite[\S4.2]{Bon.gc}
\begin{thm}[\cite{Bon.gc}]
For any compact set $K\subset X$, the intersection number
$$i:\mathcal{C}(X)\times\mathcal{C}_K(X)\to\mathbb{R}$$
is continuous.
\end{thm}

Finally, we remark that measured laminations also give geodesic currents. For a closed geodesic $\gamma$ and a measured lamination $\mathcal{L}$, two definitions of intersection number (total transverse measure $I(\mathcal{L},\gamma)$ and $i(\mathcal{L},\gamma)$) agree. Moreover, the set of measured laminations is equal to the \emph{light cone}, i.e.\ the set of geodesic currents with zero self-intersection number \cite[Prop 4.8]{Bon.gc}.

\paragraph{Measured train tracks.}
A \emph{train track} $\tau$ is an embedded $1$-complex on $X$ consisting of the set of vertices $V_{\tau}$ (which we call \emph{switches}) and the set of edges $E_{\tau}$ (which we call \emph{branches}), satisfying the following properties:
\begin{itemize}[topsep=0mm, itemsep=0mm]
    \item Each branch is a smooth path on $X$; moreover, branches are tangent at the switches.
    \item Every connected component of $\tau$ which is a simple closed curve has a unique switch of degree two. All other switches have degree at least three. At each switch $v$ if we fix a compatible orientation for branches connected to $v$, there is at least one incoming and one outgoing branch.
    \item For each component $C$ of $X-\tau$, the surface obtained from doubling $C$ along its boundary $\partial C$, has negative Euler characteristic if we treat non-smooth points on the boundary as punctures.
\end{itemize}
A \emph{measured train track} $(\tau,\omega)$ is a train track $\tau$ and a weight function on edges $\omega:E_\tau\to\mathbb{R}_{\ge0}$ satisfying the following equation for each switch $v\in V_\tau$:
$$\omega(e_1)+\dots\omega(e_i)=\omega(e'_1)+\dots+\omega(e'_j)$$
where $e_1,\dots,e_i$ are incoming branches at $v$ and $e'_1, \dots, e'_j$ are outgoing branches, for a fixed compatible orientation of branches connected to $v$. Note that the equation does not change if we choose the other compatible orientation of the branches.

It is a well-known result that each measured train track $(\tau,\omega)$ corresponds to a measured laminiation $\lambda_{\tau}.$
%We may also interpret $\mu$ as a transverse measure supported on $\tau$, by viewing $\mu$ as a weight function. More precisely, it assigns to each transverse segment $\Lambda$ the minimum --- among all transverse segments isotopic to $\Lambda$ --- of the weighted sum of intersections with the branches of $\tau$.
%$\tilde{\mu}(\Lambda')$ among all transverse segments $\Lambda'$ isotopic to $\Lambda$, where $\tilde{\mu}(\Lambda')$ is the sum of $\mu(b)$ for every branch $b$ intersecting $\tau$.  

Let $\mathcal{I}$ be an interval. A \emph{train path} is a smooth immersion $r:\mathcal{I}\to\tau\subset X$ starting and ending at a switch. We say a ray (or a multicurve, or a train track) $\gamma$ is \emph{carried} by $\tau$ if there exists a $C^1$ map $\phi: X \rightarrow X$ homotopic to the identity so that $\phi(\gamma)\subset\tau$ and the differential $d\phi_p$ restricted to the tangent line at any $p$ on the ray (or a multicurve, or a train track) is nonzero. In the case of a ray or an oriented curve, the image under $\phi$ is a train path of $\tau$, and in the case of a train track $\tau'$, $\phi$ maps a train path of $\tau'$ to a train path of $\tau$. We say a geodesic lamination $\mathcal{L}$ is carried by $\tau$ if every leaf of $\mathcal{L}$ is carried by $\tau$.

From the constructions in \S1.7 of \cite{com.tr}, we can approximate a minimal measured lamination $\mathcal{L}$ by a sequence of birecurrent measured train tracks $(\tau_n, \omega_n)$. These train tracks are deformation retracts of smaller and smaller neighborhoods of $\mathcal{L}$.

%train track is 
%More precisely, as $n \to \infty$ we have
%$$
%\lambda_{\tau_n} \to \mathcal{L}
%$$
%where $\lambda_{\tau_n}$ is the lamination corresponding to $(\tau_n, \omega_n)$.{\color{red} same reference?}
%Specifically, for each transverse segment $\Lambda$,
%$$
%\lim_{n\to\infty} \mu_n(\Lambda)= %\mu_\mathcal{L}(\Lambda).
%$$
Moreover, they are connected and satisfy the following properties: 
\begin{itemize}[topsep=0mm, itemsep=0mm]
\item They are not simple closed curves;
\item $\tau_{n+1}$ is carried by $\tau_n$ for $n \geq 1$;
\item $\omega_n$ of each branch is a positive number less than $1/{2^n}$;
\item $\mathcal{L}$ is carried by $\tau_n$ for $n \geq 1$.
\end{itemize}
For more information on train tracks, see \cite{com.tr}.

\paragraph{Topological entropy.} 
Topological entropy is an invariant of a dynamical system. It is a non-negative number  which measures complexity of dynamical system; higher entropy indicates higher complexity. More precisely, let $(Y,\sigma)$ be a dynamical system with a metric $d$. The \emph{topological entropy} of $Y$ is defined as follows. For each $T, \epsilon>0$, let $N(\epsilon, T)$ be the cardinality of the smallest finite set $\mathcal{W}\subset Y$ such that for any point $p \in Y$ there exists $p' \in \mathcal{W}$ where $d(\sigma^i(p),\sigma^i(p'))< \epsilon$ for $0\leq i \leq T$. The entropy of $Y$ is then defined to be
$$\lim \limits_{\epsilon \to 0} \limsup \limits_{T \to \infty}\frac{1}{T}\log(N(\epsilon,T)).$$
We will explain how topological entropy is involved in our problem in \S\ref{sec:carry}.

\section{Train tracks and symbolic coding}\label{sec:entropy}
In this section we use symbolic coding of train tracks to prove Theorem~\ref{thm:tail2}, which gives a criterion for convergence of two train paths.

Let $\tau=(V_\tau,E_\tau)$ be a train track carrying $\mathcal{L}$. Suppose $|E_\tau|=d$ and label the branches of $\tau$ by $b_1^{\tau}, \dots, b_d^{\tau}$. By listing the branches it traverses, we can describe a (finite or infinite) train path $r$ of $\tau$ by a (finite or infinite) word $w_{\tau}(r)$ with letters in the alphabet $B^\tau:=\{ b_i^{\tau} , 1 \leq i \leq d\}$. We can also assign such a word to a ray or a curve carried by $\tau$ if we consider the corresponding train path.

\paragraph{Symbolic coding of points at infinity.}
Assume further $\tau$ is connected and birecurrent. Let $\tilde\tau=\pi^{-1}(\tau)\subset\mathbb{H}^2$. We say a point $p$ on the circle at infinity $S^1_\infty$ is \emph{reached} by a train path of $\tau$ if some lift of the path to $\tilde\tau$ converges to $p$. Two infinite train paths are said to converge at infinity if some lifts of the paths to $\tilde\tau$ reach the same point at infinity. We have
\begin{thm}\label{thm:tail2}
Two train paths of a train track $\tau$ converge at infinity if and only if the corresponding words have the same tail. 
\end{thm}
\begin{proof}
%If two train paths go along the same branches after a point, some lifts of the paths also go along the same branches in $\tilde\tau$, and hence reach the same point at infinity.
One direction is obvious. For the other direction, let $X_1$ be the smallest subsurface containing $\tau$. Suppose $r_1$ and $r_2$ are two converging train paths. By definition, there exist lifts $\widetilde{r_1}$ and $\widetilde{r_2}$ of $r_1$ and $r_2$ to $\widetilde\tau$ so that they converge to the same point $Q$ at infinity. Assume the starting point of $\widetilde{r_i}$ is $P_i$ for $i=1,2$. We view $\widetilde{r_i}$ as an \emph{oriented} path from $P_i$ to $Q$. 

 It is easy to see that, in fact, $\widetilde{r_1}$ and $\widetilde{r_2}$ are contained in a single connected component $\widetilde{X_1}$ of $\pi^{-1}(\Int(X_1))$. We aim to prove that $\widetilde{r_1}$ and $\widetilde{r_2}$ share the same branches after a point. 
\begin{figure}[ht!]
    \centering
    \includegraphics[scale=0.6]{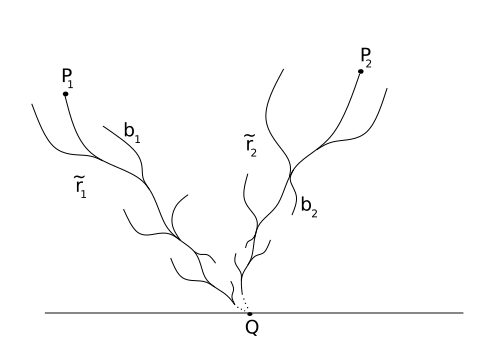}
    \caption{Converging train paths}
    \label{tr1}
\end{figure}

Suppose otherwise. We will repeatedly use the following fact: there is no embedded bigon in $\mathbb{H} \cup S_{\infty}^1$ whose boundary is contained in $r_1 \cup r_2$ \cite[Prop.\ 1.5.2]{com.tr}. We prove the following sequence of claims:
%By a branch in $\mathbb{H}$ we mean a preimage of a branch in $S_1$.
\begin{clmA} 
$\widetilde{r_1}$ and $\widetilde{r_2}$ are disjoint.
\end{clmA}
Indeed, between any intersection and $Q$, $\widetilde{r_1}$ and $\widetilde{r_2}$ bound a nonempty bigon. Therefore, we may assume $\widetilde{r_2}$ is on the left of $\widetilde{r_1}$ as we go toward $Q$ (see Figure~\ref{tr1}).

\begin{clmA}
We can connect $P_1$ and $P_2$ in $\widetilde{X_1}$ by a sequence of branches of $\widetilde{\tau}$.
\end{clmA}
Let $\widetilde{\tau}_*$ be the connected component of $\widetilde\tau$ containing $P_1$. The geodesic segment $\overline{P_1P_2}$ projects to a geodesic segment contained in $X_1$. Each complementary region of $\tau$ in $X_1$ is a disk, a cylinder with one boundary component in $\partial X_1$, or a cylinder with two boundary components formed by branches of $\tau$, at least one of them non-smooth. If the projection of $\overline{P_1P_2}$ intersects the core curve of a cylinder of this last type, then $Q$ must be the end point of a lift of the core curve. Since we can represent train paths by smooth curves whose geodesic curvature is uniformly bounded above by a small constant, we may assume $\widetilde{r_1}$ and $\widetilde{r_2}$ are within bounded distance from each other. Now both $r_1$ and $r_2$ are recurrent, so we obtain an immersed cylinder in $X_1$ whose boundary components are train paths. But this lifts to a bigon with both vertices at infinity..

Therefore, we can homotope the projection of $\overline{P_1P_2}$ rel end points to a (possibly non-smooth) path formed by the branches of $\tau$. This lifts to a path between $P_1$ and $P_2$ formed by branches of $\widetilde{\tau}_*$. Note that this may not necessarily be a train path. From now on, by $\overline{P_1P_2}$ we mean this path consisting of branches.

%Note that $g\widetilde{\tau}_1 \cap \widetilde{\tau}_1=\varnothing$ when $g \in \pi_1(X)$ is not homotopic to a closed curve in $X_1$. Let $h \in \pi_1(S)$ be the closed curve that a segment $\overline{P_1P_2}$ between $P_1$ and $P_2$ is projected to it. Then we can see $Q$ is in $\widetilde{\tau} \cap h\widetilde{\tau}$. So projection of $\overline{P_1P_2}$ on $S$ is homotopic to a closed curve in $S_1$. Each complementary region of $\tau$ in $S_1$ is a disk or an annulus with a boundary in $\partial S_1$. Therefore, every closed curve in $S_1$ is homotopic to a closed curve constructed from the branches of $\tau$ and it implies we can connect $P_1$ and $P_2$ by the branches of $\widetilde{\tau}$. From now on by $\overline{P_1P_2}$ we mean this path which is made of branches. Note that $\overline{P_1P_2}$ is not necessarily a trainpath.  

Let $b$ be a branch attached to $\widetilde{r_i}$. Then $b$ is called an \emph{incoming branch} (resp.\ \emph{outgoing branch}) if it is smoothly connected to the tail (resp.\ head) of a branch of $\widetilde{r_i}$ (recall that $\widetilde{r_i}$ is oriented, and so are its branches). In Figure~\ref{tr1}, for example, $b_1$ is an incoming branch and $b_2$ is an outgoing branch. 

\begin{clmA} \label{inf}
There are either infinitely many branches attached to $\widetilde{r_1}$ on its left, or infinitely many branches attached to $\widetilde{r_2}$ on its right.
\end{clmA}
Suppose otherwise. Then there is no branch on the left of $\widetilde{r_1}$ nor on the right of $\widetilde{r_2}$ after a point. This implies that the projection of $\widetilde{r_i}$ to $X_1$ after that point is a closed curve $\gamma_i$. Moreover, $\gamma_1$ and $\gamma_2$ must be homotopic. The region bounded between them is a cylinder with smooth boundary, which cannot appear for a train track.

%Let $\widetilde{\alpha}$ be the lift of $\alpha$ in $\mathbb{H} \cup S_{\infty}^1$ containing  $Q$. It is a semicircle orthogonal to the real line. A local neighbourhood of $\widetilde{\alpha}$ is divided into two parts, one is corresponding to $S_1$ and the other  is corresponding to $S\backslash S_1$. We can see $\widetilde{r_1}$ and $\widetilde{r_2}$ are on the side related to $S_1$. Note that $\widetilde{r_2}$ and $\widetilde{\alpha}$ are on the same side of $\widetilde{r_1}$, the side which we assumed does not have any attached branches after a point.

%So $\widetilde{r_2}$ is between $\widetilde{\alpha}$ and $\widetilde{r_1}$ and is projected to $R$ which is a contradiction and Claim~\ref{inf} is proved.

Without loss of generality, we may assume there are infinitely many branches attached to $\widetilde{r_1}$ on its left. Let $T \subset \mathbb{H} \cup S_{\infty}^1$ be the region bounded by $\widetilde{r_1}, \widetilde{r_2}$ and $\overline{P_1P_2}$. Given a branch $b$ inside $T$ attached to $\widetilde{r_1}$, we can extend it (from the end not attached to $\widetilde{r_1}$) to a train path $r_b$ until we hit the boundary of $T$. We have the following two cases.

\noindent\textbf{Case 1.}\quad There are infinitely many incoming branches. 
    
Given an incoming branch $b$, if $r_{b}$ hits $\widetilde{r_2}$, then $r_b$ and portions of $\widetilde{r_1}$ and $\widetilde{r_2}$ form a bigon with $Q$ as a vertex (see Figure~\ref{bgn1.a}), which is impossible. Moreover, given two incoming branches $b_1, b_2$, we must have $r_{b_1} \cap r_{b_2}= \varnothing$, for otherwise we would have a bigon (see Figure~\ref{bgn1.b}). 
\begin{figure}[ht]
\begin{subfigure}{.33\textwidth}
  \centering
  % include first image
  \includegraphics[width=.8\linewidth]{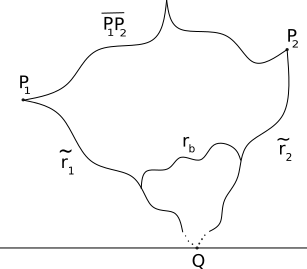}  
  \caption{}
  \label{bgn1.a}
\end{subfigure}
\begin{subfigure}{.33\textwidth}
  \centering
  % include second image
  \includegraphics[width=.8\linewidth]{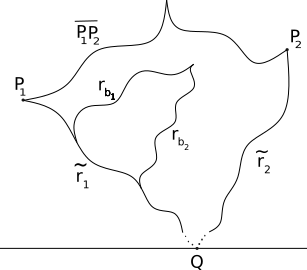}  
  \caption{}
  \label{bgn1.b}
\end{subfigure}
\begin{subfigure}{.33\textwidth}
  \centering
  % include first image
  \includegraphics[width=.8\linewidth]{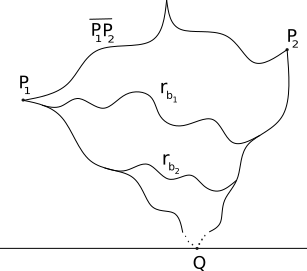}  
  \caption{}
  \label{bgn2}
\end{subfigure}
\caption{Different types of bigons}
\label{bgn1}
\end{figure}

Therefore all the extended train paths hit $\overline{P_1P_2}$. But this is impossible, as $\overline{P_1P_2}$ consists of finitely many branches.

\noindent\textbf{Case 2.}\quad There are infinitely many outgoing branches.

The arguments are similar. Given two outgoing branches $b_1, b_2$, we must have $r_{b_1} \cap r_{b_2}= \varnothing$, for otherwise we have a bigon. Moreover if $r_{b_1}$ and $r_{b_2}$ both hit $\widetilde{r_2}$ then either they bound a bigon (see Figure~\ref{bgn2}) or one of them bounds a bigon with $Q$ as a vertex. Therefore, similar to the previous case, all the extended train paths hit $\overline{P_1P_2}$, again impossible.
\iffalse
\begin{figure}[ht!]
     \centering
    \includegraphics[scale=0.8]{bgn2.png}
    \caption{Another type of bigons}
    \label{bgn2}
\end{figure}
\fi
\end{proof}
As a consequence, for every point $Q$ at infinity that is reachable by $\tau$, we may assign a symbolic coding by choosing any train path reaching $Q$. This coding is only well-defined up to the equivalence relation of having the same tail, and is $\pi_1(X)$-invariant (i.e.\ the equivalence class of words we can assign for $Q$ is the same as that of $\gamma\cdot Q$ for any $\gamma\in\pi_1(X)$). This is very much reminiscent of the classical cutting sequences for the modular surface, or boundary expansion in general (see e.g.\ \cite{simple}).

\section{Entropy and closed curves carried by a train track}\label{sec:carry}
In this section, we show that there exist inadmissible words with arbitrarily small transverse measure.

Let $\mathcal{L}$ be a minimal measured lamination and  $\tau_n$ be a sequence of train tracks approximating $\mathcal{L}$ as described in \S\ref{sec:background}. Assume each branch of $\tau_n$ has positive transverse measure $<1/2^n$, for $n \in \mathbb{N}$. For each $n$, we fix a homotopy $\phi_n$ which sends $\tau_n$ to a subset of $\tau_1$. Via $\phi_n$, each closed train path carried by $\tau_n$ determines a closed train path of $\tau_1$ and corresponding to a word by letters in $B^{\tau_1}$. A word with letters in $B^{\tau_1}$ is said to be \emph{admissible} if it is a subword of $w_{\tau_1}(l)$ for a leaf $l$ of $\mathcal{L}$, where $w_{\tau_1}(l)$ is the word corresponding to $l$. Otherwise, it is called \emph{inadmissible}.

 %Here we adopt the convention that if the starting or the ending point of a train path of $\tau_n$ is mapped to a point in the interior of a branch of $\tau_1$, we will include that branch as well to form a train path of $\tau_1$. This gives a map from the words of train paths of $\tau_n$ to those of $\tau_1$ in the obvious way.

Let $\mathcal{A}_\tau(T)$ be the set of admissible words of length less than $T$. We have:

\begin{prop}\label{prop:ent.poly}
The number $|\mathcal{A}_\tau(T)|$ of admissible words of length $<T$ has polynomial growth in $T$. More precisely, $$|\mathcal{A}_\tau(T)|< 4|E_{\tau}|^2T^{|E_{\tau}|-|V_\tau|+2}.$$
\end{prop}

\begin{proof}
\iffalse
Define $B_{\tau}(T)$ and $A_{\tau}(T)$ in the following way:

\begin{itemize}[topsep=0mm, itemsep=0mm]
\item $B_{\tau}(T)$ is the total number of words of length less than $T$ with alphabet from $B^{\tau}$ which are corresponding to a ray carried by $\tau$. 
\item $A_{\tau}(T)$ is the number of the same set of words with additional property of being admissible.
\end{itemize}

We can see $B_{\tau}$ increases exponentially. On the other hand, we now show that $A_{\tau}(T) < 4n^2T^2\binom{T+n-1}{n-1}$ which is a polynomial in $T$, of degree $n+1$ via the following argument.\\
\fi
We can describe each simple geodesic segment $\gamma$ carried by $\tau$ by all of the following information:
\begin{itemize}[topsep=0mm, itemsep=0mm]
    \item A function $f: E_{\tau} \to \mathbb{N}\cup\{0\}$ which assigns to each branch $e$ the number of times $\gamma$ passes through $e$;
    \item Two points, each on one side of a branch. These points are indicating the end points of $\gamma$.
    \item Two numbers $0<s,k<T$. Subarcs of $\gamma$ in each branch $e \in E_\tau$ have an order,  because $\gamma$ is simple and arcs do not intersect. Two numbers $s,k$ are the place of the end points in this order.
\end{itemize}   
We can construct $\gamma$ from these data uniquely (if there exists any $\gamma$ with these properties).  For each $e$, we draw $f(e)$ arcs parallel to $e$ in a neighborhood of $e$ and at each vertex we  connect incoming and outgoing arcs (except the two end points of $\gamma$) without making any intersection. Therefore, there is an injective map from $\mathcal{A}_{\tau}(T)$ to the set of multiples $(f, P_1,Q_1,s,k)$. The cardinality of this set of multiples is at most (the number of functions $f$)$\times 2|E_\tau|\times 2|E_\tau| \times T \times T$. 

The function $f$ satisfies the following equation for each switch $v\in V_\tau$:
\begin{equation}\label{f.eq}
f(e_1)+\dots+f(e_i)-f(e'_1)-\dots-f(e'_j)=m
\end{equation}
where $e_1,\dots,e_i$ are incoming branches at $v$ and $e'_1, \dots, e'_j$ are outgoing branches, for a fixed compatible orientation of branches connected to $v$ and $m \in \{ -2,-1,0,1,2 \}$ depending on the number of endings points we have on these branches (on the side connected to $v$).

Note that value of $f$ on each branch is less $T$. Equations \ref{f.eq} are independent (see \cite[Lamma 2.1.1]{com.tr}). Therefore, the number of functions $f$ satisfying these properties is at most $T^{|E_{\tau}|-|V_{\tau}|}$.
\end{proof}

 Let $b_n$ be a branch of $\tau_n$ and $\mathcal{K}_{b_n}$ the set of finite words with letters in $B^{\tau_1}$ corresponding to the closed train paths carried by $\tau_n$ and containing $b_n$. Subset $\mathcal{K}_{b_n}(T) \subset \mathcal{K}_{b_n}$ contains the words of length $<T$.
 \begin{lm}\label{lm:2admisible}
 There are at least two admissible words in $\mathcal{K}_{b_n}$ whose train paths are first return paths from $b_n$ to $b_n$.
 \end{lm}
 
 \begin{proof}
It is equivalent to say the admissible words in $\mathcal{K}_{b_n}$ are not generated by a single word. We prove it by contradiction. Assume they are. Then the infinite admissible words are generated by a single word $w \in \mathcal{K}_{b_n}$, because each leaf of $\mathcal{L}$ is recurrent and come back to $b_n$ infinitely many times. Therefore, image of each leaf in $\tau_1$ via $\phi_n$ corresponds to the word constructed from iteration of $w$. By Theorem~\ref{thm:tail2}, all leaves converge to the closed geodesic corresponding to $w$, which is a contradiction. 
 \end{proof}

\begin{cor} \label{cor:ent.exp}
 The number $|\mathcal{K}_{b_n}(T)|$ of closed train paths containing $b_n$ with length $<T$ has exponential grows in $T$.
\end{cor}
\begin{proof}
Let $a_1, \dots, a_k$ $k \geq 2$, be the words corresponding to the first return paths from $b_n$ to $b_n$ by Lemma \ref{lm:2admisible}. All the words constructed from these blocks are in $\mathcal{K}_{b_n}$. Therefore, $|\mathcal{K}_{b_n}(T)|> k^{cT}$ for a constant $c$.
\end{proof}

\begin{cor} \label{cor:inadmissible}
There is an inadmissible word $w \in \mathcal{K}_{b_n}$ with transverse measure $< 1/2^n$. 
\end{cor}
\begin{proof}
Let $a_1, \dots, a_k$ $k \geq 2$, be the words corresponding to the first return paths from $b_n$ to $b_n$ by Lemma \ref{lm:2admisible}. The set of words constructed using blocks $a_1, \dots, a_k$ has exponential grows. Therefore, from Proposition \ref{prop:ent.poly}, there must be an inadmissible word among these words.
 
 Let $w_n=a_{i_1}a_{i_2}\dots a_{i_s}$ be the shortest (smallest $s$) inadmissible word constructed from the blocks $a_1, \dots, a_k$. Then $a_{i_1}a_{i_2}\dots a_{i_{s-1}}$ is admissible, and thus can be represented by a segment of a leaf. The $a_{i_s}$ can also be represented by a segment of a leaf; connecting these two segments possibly creates a crossing over the branch $b_n$. Thus what we obtain is a segment $l_n$ representing $w_n$ with transverse measure less than $1/2^n$, as the crossing happens in a branch of $\tau_n$.
\end{proof}

\paragraph{Interpretation by entropy.}
Proposition~\ref{prop:ent.poly} and Lemma \ref{prop:ent.poly} imply the entropy of leaves of a measured lamination is zero and the entropy of closed curves carried by a train track is positive.
% {\color{red}estimate? goes to zero?}
 
To see this, for the former, we define a metric $d$ on the space of infinite words with letters in $B^\tau$ in the following way:
$$
d(w, w')=2^{-i}
$$
where $i$ is the first index that $w$ and $w'$ have different letters. The shift map $\sigma$ which removes the first letter of each word acts on this space. The set of admissible infinite words $\mathcal{A}:=\mathcal{A}_{\tau}(\infty)$ is $\sigma$-invariant, and thus gives us a dynamical system.
 
 Let $N(\epsilon, T)$ be the cardinality of the smallest finite set $\mathcal{W}\subset\mathcal{A}$ such that for any word $w$ there is $w' \in \mathcal{W}$ where $d(\sigma^i(w),\sigma^i(w'))< \epsilon$ for $0\leq i \leq T$. It means the first $T+c$ letters of $w$ and $w'$ are the same, where $c$ is a constant depending on $\epsilon$. Therefore, $N(\epsilon, T)<|\mathcal{A}_{\tau}(T+c)|$. 
 From Proposition~\ref{prop:ent.poly} the topological entropy of $\mathcal{A}$ is zero.
 
 Similarly, we can define distance function $d$, shift function $\sigma$ on the space of bi-infinite words.
 For each closed train path of $\tau_1$ (specifically for each element of $\mathcal{K}_{b_n}$), by repeating the word corresponding to it, we obtain a bi-infinite word. The set of words obtained from $\mathcal{K}_{b_n}$ is called $\widetilde{\mathcal{K}_{b_n}}$. It gives us a dynamical system.
 
 Let $M(\epsilon, T)$ be the cardinality of the smallest finite set $\mathcal{Y}\subset \widetilde{\mathcal{K}_{b_n}}$ such that for any word $y$ there is $y' \in \mathcal{Y}$ where $d(\sigma^i(y),\sigma^i(y'))< \epsilon$ for $0\leq i \leq T$. We can see $M(\epsilon, T)> |\mathcal{K}_{b_n}(T)|$ which has exponential grows in $T$, from Cor.\ \ref{cor:ent.exp}. Therefore, the topological entropy of closed curves carried by the train track is positive, as these contain $\mathcal{K}_{b_n}$.

\section{The halo of a measured lamination}\label{sec:general_proof}

In this section, we discuss some properties of the halo of a measured lamination, and give a proof of the main result Theorem~\ref{mainprime}.

Recall that given a measured lamination $\mathcal{M}$ of $\mathbb{H}^2$, the halo of $\mathcal{M}$, denoted by $h\mathcal{M}$ is the subset of $S^1$ consisting of endpoints of exotic rays. The following proposition states that any geodesic ray ending in $h\mathcal{M}$ is in turn exotic:

\begin{prop}\label{prop:point_at_infinity}
Suppose $r_1:[0,\infty)\to \mathbb{H}$ is a piecewise smooth ray and $r_2$ a geodesic ray. Assume further $r_1$ and $r_2$ reach the same point at infinity. If $I(\mathcal{M},r_1)<\infty$ then $I(\mathcal{M},r_2)<\infty$.
\end{prop}
\begin{proof}
Let $r_0$ be the (finite) geodesic segment between $r_2(0)$ and $r_1(0)$. Consider leaves intersecting $r_2$ but not $r_1$. Since $r_2$ is geodesic, each of these leaves intersects $r_2$ only once, and must also intersect $r_0$. In particular, $I(\mathcal{M},r_2)\le I(\mathcal{M},r_1)+I(\mathcal{M},r_0)<\infty$, as desired.
\end{proof}

%Old version of Prop above.
%\begin{prop}\label{prop:point_at_infinity}
%Suppose $r_1:[0,\infty)\to \mathbb{H}$ is a piecewise smooth ray and $r_2$ a geodesic ray. Assume the hyperbolic distance $d(r_1(t),r_2(t))$ is bounded independent of $t$. If $\tilde\mu_{r_1}(\infty)<\infty$ then $\tilde\mu_{r_2}(\infty)<\infty$.
%\end{prop}
%\begin{proof}
%Let $r_0$ be any path between $r_2(0)$ and $r_1(0)$. For any $t>0$, let $r_t$ be a path between $r_1(t)$ and $r_2(t)$; by assumption, we may assume the length of $r_t$ is bounded independent of $t$. let $R_t$ be the concatenation of $r_0$, $r_1[0,t],r_t$. Then $\tilde\mu(r_2[0,t])\le\tilde\mu(R_t)=\tilde\mu(r_0)+\tilde\mu(r_1[0,t])+\tilde\mu(r_t)$. Thus it suffices to show that $\tilde\mu(r_t)$ is bounded above independent of $t$; but this is clear.
%\end{proof}

As a consequence, if $r_1$ and $r_2$ are asymptotic geodesic rays, then $I(\mathcal{M},r_1)<\infty$ if and only if $I(\mathcal{M},r_2)<\infty$. In particular, any geodesic ray asymptotic to a leaf of $\mathcal{M}$ has finite intersection number, as any leaf has intersection number $0$.
%as a matter of fact, a similar argument as above implies that $r_1$ and $r_2$ have the same growth rate. In particular, any geodesic ray asymptotic to a leaf has finite transverse measure, as any leaf has transverse measure $0$. It also implies that varying the hyperbolic structure on $X$ will not change the fact that $\mu_r(\infty)<\infty$. Moreover we have:

When $\mathcal{L}=\widetilde{\mathcal{L}}$, the lift of a measured lamination $\mathcal{L}$ on $X$ to $\mathbb{H}^2$, we have the following consequence of the proposition above:
\begin{cor}\label{cor:dense}
The set of exotic vectors in $T_1(X)$ is empty or dense.
\end{cor}
\begin{proof}
The halo $h\tilde{\mathcal{L}}$ is $\pi_1(X)$ invariant, so if it is nonempty, it is also dense. Assume it is nonempty. For any $p\in X$, choose a covering map $\mathbb{H}^2\cong\mathbb{D}\to X$ so that a lift of $p$ is at the origin in the unit disk $\mathbb{D}$. The set of tangent vectors at the origin pointing towards a point in $h\tilde{\mathcal{L}}$ is dense. This is enough for our desired result.
\end{proof}
Another consequence is that varying the hyperbolic structure will not change the fact that $I(\mathcal{L},r)<\infty$, after straightening $r$ in the new hyperbolic structure. In particular, we may define \emph{halo} for a measured foliation, an entirely topological object, by endowing the surface with any hyperbolic structure.

We are now in a position to prove the main result Theorem~\ref{mainprime}:
\begin{proof}(of Theorem~\ref{mainprime})
If $\mathcal{L}$ is purely atomic, it is easy to see that $h\tilde{\mathcal{L}}=\varnothing$. On the other hand, if $\mathcal{L}$ is not purely atomic, then it has a nonatomic minimal component. An exotic ray for this component that is also bounded away from the other components is clearly an exotic ray for $\mathcal{L}$. To prove Theorem~\ref{mainprime}, it is thus sufficient to construct exotic rays that remains in a small neighborhood of this minimal component of $\mathcal{L}$. Let $\tau_n$ be a sequence of train tracks approximating this minimal component of $\mathcal{L}$ as described in \S\ref{sec:background}. We may further assume $\tau_1$ is disjoint from the other components of $\mathcal{L}$. The rays we construct come from train paths of $\tau_1$ and are thus disjoint from the other components. From now on, for simplicity, we may safely assume $\mathcal{L}$ is minimal.

We construct an exotic ray by gluing a sequence of inadmissible words of $\tau_1$, the sum of whose transverse measures is finite. We fix a homotopy $\phi_n$ which sends $\tau_n$ to a subset of $\tau_1$ and $\phi_n(\tau_n) \subset \phi_{n-1}(\tau_{n-1})$.
%Via $\phi_n$, each train path of $\tau_n$ determines a train path of $\tau_1$. Here we adopt the convention that if the starting or the ending point of a train path of $\tau_n$ is mapped to a point in the interior of a branch of $\tau_1$, we will include that branch as well to form a train path of $\tau_1$. This gives a map from the words of train paths of $\tau_n$ to those of $\tau_1$ in the obvious way.

Choose a branch $b_n \in E_{\tau_n}$, so that under the map from $\tau_n$ to $\tau_{n-1}$, a portion of $b_n$ is mapped to $b_{n-1}$. 
%Consider the train paths in $\tau_n$ which are first return paths from $b_n$ to $b_n$ and whose words are admissible. As detailed above, they give train paths of $\tau_1$. The corresponding words $a_1, \dots, a_k$ are formed by letters in $B^{\tau_1}$. Note that we obtain at least $k\ge2$ words in this way if $n$ is large enough. Let $\mathcal{B}$ be the set of words constructed using blocks $a_1, \dots, a_k$. There must be an inadmissible word in $\mathcal{B}$, as by Prop.\ \ref{prop:ent.poly}, the number of admissible words grows polynomially while the size of $\mathcal{B}$ grows exponentially.
%Let $w_n=a_{i_1}a_{i_2}\dots a_{i_s}$ be the shortest inadmissible word in $\mathcal{B}$. Then $a_{i_1}a_{i_2}\dots a_{i_{s-1}}$ is admissible, and thus can be represented by a segment of a leaf. The $a_{i_s}$ can also be represented by a segment of a leaf; connecting these two segments possibly creates a crossing over the branch $b_n$. Thus what we obtain is a segment $l_n$ representing $w_n$ with transverse measure less than $1/2^n$, as the crossing happens in a branch of $\tau_n$.
From Cor. \ref{cor:inadmissible}, there is a segment $l_n$ carried by $\tau_1$ with transverse measure $<1/2^n$ and starting and ending at $b_n$.

If we glue the segments $l_n$ for $n\ge1$ together we have an ray of finite transverse measure. Indeed, since both $l_n$ and $l_{n-1}$ can be represented by train paths starting and ending in $b_{n-1}$, connecting them create a transverse measure of at most $1/2^{n-1}$. So the total transverse measure $<2\sum1/2^{n}<\infty$. This ray is exotic. Indeed, by construction the ray intersects the lamination infinitely many times; moreover, its word $w_1w_2w_3\cdots$ contains inadmissible subwords in any tail, and by Theorem~\ref{thm:tail2}, this ray does not converge to any leaf of $\mathcal{L}$.

Finally, by passing to a subsequence, we may assume the transverse measures of $l_n$ satisfy $I(\mathcal{L},l_n)/3>\sum_{n+1}^\infty I(\mathcal{L},l_k)$. For any subsequence $\{n_k\}$ of $\mathbb{N}$, we may glue the segments $l_{n_k}$ instead. Note that $w_{n_1}w_{n_2}\cdots$ and $w_{n'_1}w_{n'_2}\cdots$ have the same tail if and only if the sequences $\{n_k\}$ and $\{n'_k\}$ have the same tail. Indeed, our assumption $I(\mathcal{L},l_n)/3>\sum_{n+1}^\infty I(\mathcal{L},l_k)$ means that if the two sequences have different tails, the tails of the corresponding words have different transverse measure. Given any subequence $\{n_k\}$, there are at most countably many subsequences of $\mathbb{N}$ having the same tail. So there are uncountably many of subsequences not having the same tail, implying that there are uncountably many exotic rays with different end points.
\end{proof}

\section{The case of a punctured torus}\label{sec:tori}
In this section, we discuss Theorem~\ref{main1} in the special case where $X$ is a hyperbolic torus with one cusp. In particular, we highlight the connections with continued fractions and Sturmian words. We also remark that the approach adopted here involves translation surfaces and their directional flows, which can be generalized to give a proof in the general case as well. We give an exposition of this in Appendix~\ref{sec:translation}.

\paragraph{Measured laminations on $X$.}
Given a hyperbolic torus $X$ with one cusp, let $\bar X$ be the torus with the cusp filled in (denote the new point by $p$). Uniformization gives a complex structure on $X$ which extends uniquely to $\bar X$. We can present $\bar X\cong\mathbb{C}/(\mathbb{Z}\oplus\mathbb{Z}\tau)$ with $p$ at the origin. Since varying hyperbolic structure does not affect our result, we may as well assume $\tau=i$.

Any measured lamination on $X$ comes from a measured foliation on $\bar{X}\cong\mathbb{C}/\mathbb{Z}^2$. More precisely, fix a real number $\theta$ and consider all lines in $\mathbb{C}$ of slope $\theta$. This gives a foliation on $\bar{X}$, and the transverse measure is given by the length measure in the perpendicular direction. Straightening the lines with respect to the hyperbolic metric on $X$ gives a geodesic lamination, which is atomic if and only if $\theta$ is rational. Thus the space of projectivized measured laminations on $X$ is identified with $\mathbb{RP}^1$ with simple closed geodesics identified with $\mathbb{QP}^1$. We can thus talk about \emph{rational} and \emph{irrational} laminations with this identification. For simplicity, we call the leaf of a lamination/foliation with slope $\theta$ a \emph{$\theta$-leaf}.

\paragraph{Symbolic coding of rays.}
Recall that in Figure~\ref{fig:symbolic_coding}, we illustrated a way of encoding geodesic rays in $X$ by choosing a quadrilateral fundamental domain. Alternatively, we present $X$ as $(\mathbb{C}-\mathbb{Z}^2)/\mathbb{Z}^2$. The open set $\mathbb{C}-\mathbb{Z}^2$ has an intrinsic hyperbolic metric, and by symmetry, horizontal line segments connecting neighboring integral lattice points are geodesics in this metric. The unit square $S=[0,1]\times[0,1]$ with this metric is thus isometric to an ideal quadrilateral. We refer to the left, bottom, right, top sides of $S$ as $J_1,I_1,J_2,I_2$.

Set $\bolda$ as the map $z\mapsto z+1$ and $\boldb$ as $z\mapsto z+i$. We consider $J_1,I_1,J_2,I_2$ as labeled by $\overline{\bolda}$ ($:=\bolda^{-1}$), $\overline{\boldb}$, $\bolda$, $\boldb$, as the squares that share these sides with $S$ are precisely the image of $S$ under the corresponding label. Given a geodesic ray $r$ in $X$, its preimage via the map $S\to X$ consists of countably many geodesic segments $r_1,r_2,r_3,\ldots$, and each segment $r_i$ ends on one of the sides of $S$. Recording the label of the sides $r_i$ lands on, we obtain a reduced infinite word $w_r$. It is easy to see that this word is the same as the one we obtain by taking a lift to $\mathbb{H}^2$ with initial point in $Q$, as in the introduction. Two geodesic rays in $X$ are asymptotic if and only if the corresponding words have the same tail.

\paragraph{Continued fraction and rational approximation.}
Let $\theta$ be a positive real number, and $[c_0;c_1,\ldots]$ its continued fraction. The continued fraction has finite length if and only if $\theta$ is rational. Let $p_k/q_k:=[c_0;c_1,\ldots,c_k]$, the $k$-th \emph{convergent} of $\gamma$. It is easy to see that $p_k/q_k\le\theta$ when $k$ is even, and $p_k/q_k\ge\theta$ when $k$ is odd. They are, in some sense, best approximations to $\theta$ (see e.g. \cite[\S II.6]{continued_fraction}):
\begin{thm}\label{thm:rational_approx}
For any integer $k\ge0$, we have $|q_k\theta-p_k|<1/q_{k+1}$. Moreover, $\min_{\substack{p,q\in\mathbb{Z}\\0<q\le q_k}}|q\theta-p|=|q_k\theta-p_k|$.
\end{thm}
Let $s\in(0,1)$, chosen so that no points in $\mathbb{Z}^2$ lie on the line of slope $\theta$ passing through $s\sqrt{-1}$. Give two complex numbers $x,y$, let $l(x,y)$ be the line segment between them. For each $k$, consider the path consists of $l(s,s+q_k+q_k\theta i)$ and $l(s+q_k+q_k\theta i,s+q_k+p_ki)$. Theorem~\ref{thm:rational_approx} implies that this path maps to a simple closed curve in $\bar X$; as a matter of fact, the following lemma guarantees that when $k$ is large enough, we can homotope the curve to a flat geodesic in $X$:
\begin{lm} \label{lm:homotop}
Set $\epsilon_k=\begin{cases}s&k\text{ is odd}\\1-s&k\text{ is even}\end{cases}$. If $1/q_k<\epsilon_k$, then there exists $s'\in(0,1)$ so that no integer lattice points lie in the region bounded by the lines $y=\theta x+s$, $y=(p_k/q_k)x+s'$, $x=0$, and $x=q_k$.
\end{lm}
\begin{proof}
We will consider the case where $k$ is even; the other case is similar. In particular, $p_k/q_k\le\theta$. For each integer $0\le l\le q_k-1$, let $y_l=s+(p_k/q_k)l\mod 1$. Then $y_0,\ldots,y_{q_k-1}$ are $q_k$ points of equal distance $1/q_k$ on $(0,1)$. Suppose $y_{l_0}$ is the largest of them. Let $s'=s+l_0(\theta-p_k/q_k)\in(0,1)$. It is easy to see this $s'$ satisfies the conditions we want. Indeed, when $l\neq l_0$, $1-y_l>1/q_k$; since points on the line $y=\theta x+s$ are at most $1/q_k$ above the points on the line $y=(p_k/q_k)x+s$, we have $\theta l+s$, $(p_k/q_k)l+s$ and $(p_k/q_k)l+s'$ all have the same integral part.
\end{proof}
Hence we may approximate the leaf of any irrational lamination with simple closed geodesics in the flat metric; moreover, for any $\theta$-leaf, when $k$ is large enough, the first $q_k$ letters in the word of the leaf form the word of a $p_k/q_k$-leaf. For simplicity, we call the word of any $\theta$-leaf a \emph{$\theta$-word}.

It would thus be useful to know the words of simple closed geodesics. For this we quote \cite{simple_word}:
\begin{thm}\label{thm:simple_word}
Up to permutations of generators which interchange $\bolda$ and $\boldb$, $\bolda$ and $\overline\bolda$, or $\boldb$ and $\overline\boldb$, a simple word $w$ in $\pi_1(X)$ is up to a cyclic permutations either $\bolda$ or $\bolda\overline\boldb \overline\bolda \boldb$ or
$$w=\bolda^{n_1}\boldb\bolda^{n_2}\boldb\cdots \bolda^{n_k}\boldb$$
where $\{n_1,\ldots,n_k\}\subset\{n,n+1\}$ for some $n\in\mathbb{Z}^{+}$.
\end{thm}
Note that for the word $w=\bolda^{n_1}\boldb\cdots \bolda^{n_k}\boldb$ the corresponding element in $\pi_1(\bar X)$ is $\bolda^{\sum_{j=1}^kn_j}\boldb^{k}$, which gives a simple closed curve of slope $k/\sum_{j=1}^kn_k\in[1/(n+1),1/n]$. For simplicity, we simply write $w$ above as $(n_1,\ldots,n_k)$, or $(n_1,\ldots,n_k)_{\bolda,\boldb}$ if we want to specify the letters in order, and call each segment $\bolda^{n_i}\boldb$ an ($n_i$-)block of $w$.

Suppose $r=p/q$ is a rational number $>1$. For any leaf of slope $r$ starting at a point on $J_1$, the first few letters in its word must be a few $\boldb$, followed by an $\bolda$, coming back to $J_1$ again. Hence the word for this simple closed geodesic has exactly the form $w$ described in Theorem~\ref{thm:simple_word}, with $\bolda,\boldb$ interchanged. The positive integer $n$ as in the theorem is uniquely specified unless $r=n$. In any case, the number of $n$-blocks is given by $s=(n+1)q-p$ and the number of $(n+1)$-blocks is given by $t=p-nq$. Each block $\boldb^{n_j}\bolda$ represents a sequence of segments in $S$, starting with a segment starting on $J_1$ and ending with a segment ending on $J_2$. Different blocks are disjoint from each other. We have
\begin{lm}
\begin{enumerate}[label=\normalfont{(\arabic*)}, topsep=0mm, itemsep=0mm]
    \item The starting points of the $(n+1)$-blocks are on top of those of the $n$-blocks, while the end points are switched.
    \item The relative position of the starting points of $n$-blocks is the same as the relative position of the end points; the same is true for $(n+1)$-blocks.
\end{enumerate}
\end{lm}
\begin{figure}[ht!]
    \centering
    \includegraphics{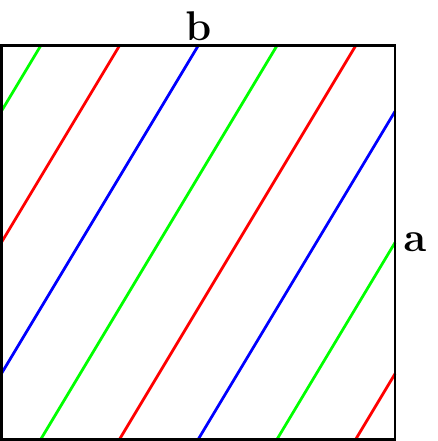}
    \iffalse
    \begin{tikzpicture}[scale=4,thick]
    \draw[red] (0,0.5)--(0.3,1);
    \draw[red] (0.3,0)--(0.9,1);
    \draw[red] (0.9,0)--(1,0.166666667);
    \draw[blue] (0,0.166666667)--(0.5,1);
    \draw[blue] (0.5,0)--(1,0.833333333);
    \draw[green] (0,0.833333333)--(0.1,1);
    \draw[green] (0.1,0)--(0.7,1);
    \draw[green] (0.7,0)--(1,0.5);
    \draw (0,0)--(0,1)--(1,1)--(1,0)--cycle;
    \node at (0.5,1.05) {$\boldb$};
    \node at (1.07,0.5) {$\bolda$};
    \end{tikzpicture}
    \fi
    \caption{$1$- and $2$-blocks in $S$}
    \label{fig:blocks}
\end{figure}
The proof is trivial; see Figure~\ref{fig:blocks} for an example. This suggests the following algorithm to produce $r$-words for $r=p/q>1$:
\begin{alg}\label{alg:simple_word}
\begin{enumerate}[topsep=0mm, itemsep=0mm]
    \item Set $s=(n+1)q-p$ and $t=p-nq$. Suppose $sp:\{1,\ldots,s+t\}\to\{n,n+1\}$ and $ep:\{1,\ldots,s+t\}\to\{n,n+1\}$ record the ordered list of blocks at the starting points and the end points respectively. That is, $sp(j)=n+1=ep(s+j)$ for $1\le j\le t$ and $sp(k+t)=n=ep(k)$ for $1\le k\le s$. Set $w$ to be the empty word. Let $l=l_1$ be any integer between $1$ and $s+t$.
    \item Set $w=w\boldb^{sp(l)}\bolda$. If $1\le l\le t$, set $l=s+l$; otherwise, set $l=l-t$.
    \item If $l=l_1$, output $w$; otherwise repeat Step 2.\qed
    \end{enumerate}
\end{alg}

Note that a different choice of $l_0$ gives a cyclic rearrangement of blocks. Together with the idea of rational approximation of irrational leaves, we have
\begin{prop}\label{prop:sturmian}
Let $\theta=[c_0;c_1,c_2,\ldots]$ be an irrational number $>1$, $p_k/q_k$ its $k$-th convergent, and $n=\lfloor\theta\rfloor$. Let $w$ be the word of a (half) leaf of the foliation of slope $\theta$ starting at a point on $J_1$. Then
\begin{enumerate}[label=\normalfont{(\arabic*)}, topsep=0mm, itemsep=0mm]
    \item $w=\boldb^{n_1}\bolda\boldb^{n_2}\bolda\cdots=(n_1,n_2,\dots)$, where $\{n_1,n_2,\ldots\}\subset\{n,n+1\}$;
    \item For all $k$ large enough, $w(k):=(n_1,n_2,\ldots,n_{q_k})$ gives a $p_k/q_k$-word.
\end{enumerate}
Furthermore, any $p_k/q_k$-word can be extended to a $\theta$-word.
\end{prop}
This proposition provides a lot of information about irrational words, and will be useful in our construction below. Similar statements can be made for irrational numbers in $(0,1)$, or one notices that any $1/\theta$-word can be obtained from a $\theta$-word by interchanging $\bolda$ and $\boldb$.

\paragraph{Inadmissible words.}
Proposition~\ref{prop:sturmian} suggests that some subwords can never appear in a $\theta$-word. For example, if $\theta\in(n,n+1/2)$, since $w(k)$ is a $p_k/q_k$-word for all $k$ sufficiently large, and the number of $n$-blocks $s_k$ is larger than or equal to the number of $(n+1)$-blocks $t_k$ in $w(k)$, there does not exist two consecutive $(n+1)$-blocks in $w(k)$ and thus $w$ itself. We call a word ($\theta$-)\emph{inadmissible} if it can never appear in any $\theta$-word, and ($\theta$-)\emph{admissible} otherwise. We have
\begin{lm}
There exists a $\theta$-inadmissible word consisting of $q_k$ blocks for all $k\ge2$, so that by changing the last block from $n$ to $n+1$, or $n+1$ to $n$ it becomes admissible.
\end{lm}

\begin{proof}
We obtain these words by the following algorithm:
\begin{alg}\label{alg:tori_inadmissible_word}
\begin{enumerate}[topsep=0mm, itemsep=0mm]
    \item If $k$ is even, set $l_1=q_k$; otherwise set $l_1=1$.
    \item Run Algorithm~\ref{alg:simple_word} with $r=p_k/q_k$ and $l_1$. Let $w$ be the output.
    \item Change the last block of $w$ from an $n$-block to an $(n+1)$-block, or from an $(n+1)$-block to an $n$-block. Output the modified word.\qed
    \end{enumerate}
\end{alg}
Note that when $k$ is even, the output starts and ends with an $n$-block, while when $k$ is odd, the output starts and ends with an $(n+1)$-block; moreover it is \emph{not} a $p_k/q_k$-word. It remains to show it is $\theta$-inadmissible. For this, it is enough to show that the output is $p_l/q_l$-inadmissible for all large $l$. Indeed, when $k$ is even, $p_{k}/q_{k}<p_l/q_l$ for all $l$ large. Moreover, by properties of continued fraction, $|q_{k}\frac{p_l}{q_l}-p_{k}|<1/q_k\le1/2$. When $\theta\in(n+1/(m+1),n+1/m)$, the word constructed above starts with $m$ $n$-blocks. If we start with $m$ $n$-blocks, and flow along lines of slope $p_l/q_l$ for $q_k$ blocks, we will end up at a point above the starting point, and connecting the segment between this point and the starting point creates a simple closed curve homotopic to a $p_k/q_k$-leaf. In particular, the first $q_k$ blocks of any $p_l/q_l$-word starting with $m$ $n$-block form a $p_{k}/q_{k}$-word. Thus the word constructed in the algorithm above cannot appear. The argument for $k$ odd is entirely analogous.
\end{proof}

\begin{prop}
For any $k\ge2$, there exists an $\theta$-inadmissible word $w_{k}$ of length $\le 2(p_k+q_k)$ so that a representative of this word has transverse measure $\le C/q_{k}$ on $X$, where $C$ is a constant independent of $k$. Moreover, we can connect the starting point of the representative of $w_{k+2}$ and the end point of that of $w_k$ by a line segment of length $<1/q_k$ in the flat metric of $X$ without passing the integral lattice points. 
\end{prop}
\begin{proof}
We obtain these words by the following algorithm:
\begin{alg}\label{alg:tori_segments}
\begin{enumerate}[topsep=0mm, itemsep=0mm]
    \item Run Algorithm~\ref{alg:tori_inadmissible_word} for $\theta$ and $k$. Let $w_k$ be the output.
    \item When $k$ is even, run Algorithm~\ref{alg:simple_word} for $r=p_k/q_k$ and $l_1=t_k+1$ partially, and stop when $l=q_k$ instead of $l=l_1$ in Step 3 there. When $k$ is odd, run Algorithm~\ref{alg:simple_word} for $r=p_k/q_k$ and $l_1=t_k$ partially, and stop when $l=1$. Let the output be $w$.
    \item Let $\tilde w$ be the word $w$ with the first block removed. Set $w_k=w_k\tilde w$. Output $w_k$.\qed
\end{enumerate}
\end{alg}
The word constructed has length less than twice of a $p_k/q_k$ word; the length of a $p_k/q_k$ word is $(n+1)s_k+(n+2)t_k=p_k+q_k$. It remains to construct a representative curve of the word $w_k$ satisfying the conditions. Take any $p_k/q_k$-leaf and draw it as a collection of line segments of slope $p_k/q_k$ in $S$. We can do this by putting equally distant points $\sigma_1,\ldots,\sigma_{q_k}$ on $J_1$ and $\tau_1,\ldots,\tau_{q_k}$ on $J_2$ as the starting and end points of the blocks. We construct a representative of $w_k$ as follows. Assume first $k$ is even. Start from $\sigma_{q_k}$, the lowest starting point of an $n$-block. Flow $q_{k}-1$ blocks along segments of slope $p_k/q_k$, and we arrive at the lowest starting point of an $(n+1)$-block. Go vertically to the highest starting point of an $n$-block. Flow along segments of slope $p_k/q_k$ until we arrive back at $\sigma_{q_k}$. Compared with flowing along lines of slope $\theta$, after the full cycle of a $p_k/q_k$-word, the vertical distance is $q_k\theta-p_k<1/q_k$. Thus after the full cycle of $w_k$, the total vertical shift is $<3/q_k$. This gives that the representative we constructed has transverse measure $\le C/q_k$, for some constant $C$ only dependent on $\theta$.

Finally, as $\sigma_{q_k}$ and the starting point of the representative of $w_{k+2}$ both has $y$-coordinate in $(0,1/q_k)$, we can connect then by a vertical segment of length $<1/q_k$ without passing through integral lattice points. The case when $k$ is odd is entirely analogous.
\end{proof}

We remark that each iteration of Algorithms~\ref{alg:tori_inadmissible_word} and \ref{alg:tori_segments} only depends on the rational number $r=p_k/q_k$ and the oddity of $k$, but not on $\theta$ directly. Now we are in a position to prove Theorem~\ref{thm:tori}, a special case of Theorem~\ref{main1}.
\begin{proof}(of Theorem~\ref{thm:tori})
Let $w_{k}$ be the word constructed in Algorithm~\ref{alg:tori_segments} with $k\ge2$, with the representative curve $C_k$ constructed in the proof there. Given $(i_k)$ Let $w=w_{i_1}w_{i_2}\cdots$, the concatenation of all $w_{i_k}$ for $k\ge1$ in order. A representative of $w$ can be constructed by connecting the end point of $C_{i_k}$ and the starting point of $C_{i_{k+1}}$ by a vertical segment of length $<1/q_{i_k}$. Therefore, the total transverse measure of this representative bounded above by $C'\sum_{k=1}^\infty1/q_{i_k}<\infty$. Finally, any tail of this word contains an inadmissible word; therefore, it is not asymptotic to any leaf.
\end{proof}

\begin{rmk}
\begin{enumerate}[topsep=0mm, itemsep=0mm]
    \item There is an abundance of choices in the construction above. Besides skipping some $w_k$'s, one may also connect $w_{i_k}$ and $w_{i_{k+1}}$ with arbitrarily long segments of $\theta$-words.
    \item As we remarked above, algorithms described in the construction do not depend on the irrational number $\theta$ directly, but on its convergents. In particular, the construction here can be easily converted to a computer program. To output the first few letters of an exotic word we do not need to know the exact value of $\theta$; sufficiently accurate rational approximations suffice.
\end{enumerate}
\end{rmk}

\paragraph{Alternative construction.}
We give an alternative construction of exotic rays that are distinct from those produced above. The key observation is that in $X$, any irrational lamination $\mathcal{L}$ is bounded away from the cusp. An excursion into the cusp thus has very small intersection with $\mathcal{L}$; on the other hand, a geodesic ray asymptotic to a leaf of $\mathcal{L}$ makes no excursion into the cusp, possibly after a finite segment. We have
\begin{thm}\label{thm:tori_alt}
Given an irrational lamination $\mathcal{L}$ of slope $\theta$, there exist uncountably many exotic rays bounded away from $\mathcal{L}$.
\end{thm}
Here, given a geodesic ray $r:[0,\infty)\to X$, we say it is \emph{bounded away} from $\mathcal{L}$ if
$$\limsup_{t\to\infty} d(r(t),\mathcal{L})>0.$$
On the other hand, it is easy to see that the construction given in the proof of Theorem~\ref{thm:tori} produces rays that are getting closer and closer to $\mathcal{L}$.
\begin{proof}
Again, we may assume $\theta>1$. For each $k\in\mathbb{N}$, let $p_k/q_k$ be the $k$-th convergent of $\theta$. Let $w_k$ be the $p_k/q_k$-word obtained by running Algorithm~\ref{alg:simple_word} with $l=s_k+t_k=q_k$. A representative curve of $w_k$ is a line segment of slope $p_k/q_k$ starting at a point slightly above the origin. Let $w=\overline{\boldb}\overline{\bolda}\boldb\bolda$. This corresponds to a loop around the cusp. Consider the word $w_1w^{k_1}w_2w^{k_2}w_3w^{k_3}w_4\cdots$, where $k_1,k_2,k_3,\ldots$ are positive integers. We may represent this word by segments connecting integral points (corresponding to $w_k$) with infinitesimally small loops around these integral points in between. The intersection number of this ray with $\mathcal{L}$ is bounded by $C\sum |q_k\theta-p_k|<\infty$. Moreover, it cannot be asymptotic to any leaf, as it contains $\bar{\bolda}$ and $\bar{\boldb}$ infinitely often. Or, as described above, it wraps around the cusp (an excursion into the cusp) infinitely often, and thus bounded away from $\mathcal{L}$. Uncountability again follows by considering subsequences of $\mathbb{N}$.
\end{proof}
\begin{rmk}
This construction was pointed out to us by C. McMullen, inspired by his paper \cite{local_connectivity}. There, a geodesic ray on $X$ is decomposed into simple segments with excursions into cusps in between.
\end{rmk}

\section{Complementary results}\label{sec:complementary}
In this section, we study general behavior of the intersection of a geodesic ray with $\mathcal{L}$ and prove two main Theorems \ref{thm:linear} and \ref{sublinear}.

Given a geodesic ray $r:[0,\infty)\to X$, let $I_{\mathcal{L},r}(t)$ be the intersection number of the geodesic segment $r([0,t])$ with $\mathcal{L}$. Clearly $I_{\mathcal{L},r}(t)$ is a nonnegative and increasing continuous function in $t$. Given two nonnegative functions $f,g$ defined on $[0,\infty)$, we write $f=O(g)$ if there exist constants $t_0,c>0$ so that $f(t)\le cg(t)$ for all $t\ge t_0$, and write $f\asymp g$ if $f=O(g)$ and $g=O(f)$. We have:
\begin{thm}\label{thm:linear}
For any measured lamination $\mathcal{L}$ and geodesic ray $r$ on $X$, $I_{\mathcal{L},r}(t)=O(t)$. Moreover, for almost every choice of initial vector $r'(0)$ with respect to the Liouville measure, $I_{\mathcal{L},r}(t)\asymp t$. More precisely, $\lim_{t \rightarrow \infty} I_{\mathcal{L},r}(t)/{t}= \ell(\mathcal{L})/(\pi^2(2g-2+n))$. 
\end{thm}
Here $\ell(\mathcal{L})$ is the \emph{length} of $\mathcal{L}$; for precise definition see \S \ref{sec:background}. This theorem implies that the set of exotic vectors (i.e.\ initial vectors of exotic rays) has measure zero. However, recall that it is dense in $T_1(X)$ by Corollary~\ref{cor:dense}.

We also show that any growth rate between finite and linear is achieved:
\begin{thm}\label{sublinear}
Suppose the measured lamination $\mathcal{L}$ is not purely atomic. Given a continuous increasing function $f:\mathbb{R}_+\to\mathbb{R}_+$ such that $f(0)=0$ and $\lim_{t\to\infty}f(t)/t=0$, there exists a geodesic ray $r:\mathbb{R}_+\to X$ such that $I_{\mathcal{L},r}\asymp f$.
\end{thm}
We remark that this does not imply Theorem~\ref{main1}, as here we only guarantee the existence of a geodesic ray with a specific growth rate, \emph{without} any restriction on its endpoint.

%we consider geodesic rays that have infinite intersection with $\mathcal{L}$, and prove Theorems~\ref{thm:linear} and \ref{sublinear}.

%Recall that given a geodesic ray $r:[0,\infty)\to X$, we define $\mu_r(t):=\mu_\mathcal{L}(r[0,t])$.
We start with the first part of Theorem~\ref{thm:linear}:
%{\color{red} keep it prop?}
\begin{prop}\label{linr}
For any geodesic ray $r:[0,\infty)\to X$, $I_{\mathcal{L},r}(t)=O(t)$.
\end{prop}
\begin{proof}
If $X$ has cusps, there exists a cuspidal neighborhood for each cusp so that the measured lamination is contained in $X^\circ$, the complement of these neighborhoods. Any geodesic segment contained in these neighborhoods has transverse measure zero. The set of geodesic segments on $X$ of length $a$ and intersecting $X^\circ$ is homeomorphic to $T_1(X^\circ)$ and is hence compact. Therefore, the transverse measure of such a segment is less than a constant $b$. If we split $r$ into segments of length $a$, then each segment is either contained in a cuspidal neighborhood or intersects $X^\circ$. Therefore
$$I_{\mathcal{L},r}(t) < b \lceil \frac{t}{a} \rceil.$$
This implies $I_{\mathcal{L},r}(t)=O(t)$, as desired.
\end{proof}

It would be interesting to see what initial vectors give exactly linear growth. It is obvious that if the image of $r$ is a closed curve, then the growth is linear (or constantly zero if the curve is disjoint from $\mathcal{L}$). Moreover we have
\begin{prop}\label{prop:leaf}
Assume $r:[0, \infty) \rightarrow X$ is an eventually simple geodesic ray; more precisely, there exits $t>0$ such that $r:[t, \infty) \to X$ does not intersect itself on $X$. Then $I_{\mathcal{L},r}(t)\asymp t$ or $I_{\mathcal{L},r}(t)\asymp 1$.
\end{prop}
\begin{proof}
If $r$ converges to a closed curve then the growth rate is linear. Otherwise, $r$ converges to a leaf of a lamination $\mathcal{L}_1$. So we might as well assume $r$ is part of a leaf of a minimal lamination $\mathcal{L}_1$.

If $\mathcal{L}_1$ is a minimal component of $\mathcal{L}$, we already know $I_{\mathcal{L},r}(t)\asymp 1$. If $\mathcal{L}_1$ and $\mathcal{L}$ are disjoint, then clearly $I(\mathcal{L},r)<\infty$ so $I_{\mathcal{L},r}(t)\asymp1$ as well. Thus we may assume $\mathcal{L}$ and $\mathcal{L}_1$ intersect transversely.

We claim that there exists a constant $C$ so that for any segment $L$ of length $C$ on a leaf of $\mathcal{L}_1$, $I(\mathcal{L},L)>0$. Suppose otherwise. Then there exists $C_i\to\infty$ so that a segment $L_i$ of length $C_i$ on a leaf of $\mathcal{L}_1$ is contained in a complementary region. We may as well assume the starting point $p_i$ of $L_i$ lies on a leaf of $\mathcal{L}$. Taking a subsequence if necessary, we may assume $p_i\to p$. Moreover, since each $L_i$ lies on a leaf of $\mathcal{L}_1$ and the length of $L_i$ goes to $\infty$, $p$ lies on a leaf of $\mathcal{L}_1$ and the half leaf starting at $p$ is contained entirely in a complementary region of $\mathcal{L}$. So either this half leaf of $\mathcal{L}_1$ is asymptotic to a leaf of $\mathcal{L}$, or is contained entirely in a subsurface disjoint from $\mathcal{L}$. Either contradicts our assumption that $\mathcal{L}$ and $\mathcal{L}_1$ intersect transversely.

By compactness, there exists $c>0$ so that any segment $L$ of length $C$ on a leaf of $\mathcal{L}_1$ has $I(\mathcal{L},L)>c$. Dividing the geodesic ray $r$ into segments of length $C$, we conclude that $I_{\mathcal{L},r}(t)>ct$ with a possibly smaller constant $c>0$. Together with Proposition~\ref{linr}, we have $I_{\mathcal{L},r}(t)\asymp t$.
\end{proof}

It turns out that linear growth is typical with respect to the Liouville measure $\lambda$:

\begin{thm}\label{thm:generic}
For a.e.\ geodesic ray $r$, we have $I_{\mathcal{L},r}(t) \asymp t$. Moreover,
$$\lim_{t \rightarrow \infty}\frac{I_{\mathcal{L},r}(t)}t= \frac{1}{\pi^2(2g-2+n)}\ell(\mathcal{L}).$$
\end{thm}
\begin{proof}
Recall that we denote the geodesic flow on $T_1(X)$ by $\phi_t$, which is ergodic with respect to the Liouville measure $\lambda$ \cite{erg.g.f}. Moreover, $\lambda/(\pi^2(2g-2+n))$ is a probability measure on $T_1(X)$. Define a function $f$ on $T_1(X)$ as follows: $f(v)$ is the transverse measure (with respect to $\mathcal{L}$) of the geodesic segment of unit length with initial tangent vector $v$. By continuity of $f$ and Birkhoff's ergodic theorem, for almost every $v \in T_1(X)$ we have:
\begin{equation}\label{erg}
\lim_{T \rightarrow \infty} \frac{1}{T} \int_{0}^{T} f(\phi_t(v)) dt= c, 
\end{equation}
where $c$ is a constant equal to $1/(\pi^2(2g-2+n)) \int_{T_1(X)} f d\lambda$.

We write $f \simeq g$ if the difference between $f$ and $g$ is less than a constant.

\begin{clm}\label{int.f}
The integral of $f$ along a geodesic segment is roughly the total transverse measure of the segment. In other words, we have:
$$\int_{0}^{T} f(\phi_t(v)) dt \simeq I_{\mathcal{L},r}(T).$$
\end{clm}
\begin{proof}
It is enough to prove the claim for each minimal component of $\mathcal{L}$.

If a component is an atom, i.e.\ a simple closed geodesic $\gamma$, $I_{\mathcal{L},r}(T)$ gives the number of intersections of $r([0,T])$ with $\gamma$, while the integral can be larger if $r([T,T+1])$ also intersects $\gamma$. Since the number of intersections of a unit length geodesic segment with $\gamma$ is bounded above, we have the desire relation.

If a component is not an atom, then the transverse measure induced on $r$ is absolutely continuous with respect to the length measure $ds$. Therefore, by Radon-Nikodym theorem, there is a measurable function $\tau$ on $r$ such that
$$ I(\mathcal{L},r[t_1,t_2])= \int_{t_1}^{t_2} \tau(r(t)) dt.$$
In particular, we can write $f(v)$ as $\int_{0}^{1} \tau(\phi_t(v)) dt$.  
Therefore, we have:
$$
\int_{0}^{T} f(\phi_t(v)) dt = \int_{0}^{T} \int_{0}^{1} \tau(\phi_{t+s}(v)) ds dt= \int_{0}^{1} \int_{0}^{T} \tau(\phi_{t+s}(v)) dt ds
$$
by Fubini's Theorem. Now it is easy to see
$$
\int_{0}^{1} \int_{0}^{T} \tau(\phi_{t+s}(v)) dt ds \simeq \int_{0}^{1} \int_{0}^{T} \tau(\phi_{t}(v)) dt ds= \int_{0}^{T} \tau(\phi_{t}(v)) dt= I_{\mathcal{L},r}(T).
$$
Here we used the fact that the transverse measure of geodesic segments of length at most one is uniformly bounded.
\end{proof}

Claim~\ref{int.f} and Equation~\ref{erg} imply $\lim_{t\to\infty}I_{\mathcal{L},r}(t)/t=c$. In order to find $c$ we show for a sequence $T_i \rightarrow \infty$, $\lim \limits_{i \rightarrow \infty} I_{\mathcal{L},r}(T_i)/T_i=\ell(\mathcal{L})$.

Let $\alpha_v^T=\frac1T(\phi_t(v))_*dt$ be $1/T$ of the pushforward of the Lebesgue measure $dt$ on $[0,T]$ via $\phi_t(v)$. From Equation~\ref{erg}, we can see that for a generic $v\in T_1(X)$, $\alpha_v^T$ converges to $\lambda/(\pi^2(2g-2+n))$ when $T \rightarrow \infty$. Moreover, if $v$ is generic, the geodesic ray $\phi_t(v)$ is recurrent, and so there exists a sequence $T_i\to\infty$ such that $\phi_{T_i}(v)\to v$. By Anosov closing lemma \cite{erg.g.f}, there is a closed geodesic $\gamma_i$ very close to the geodesic segment from $v$ to $\phi_{T_i}(v)$. Therefore, $\frac{1}{T_i}\gamma_i$ converges to $\lambda/(\pi^2(2g-2+n))$ as a sequence of geodesic currents. Moreover, we have
$$\frac{1}{T_i} \int_{0}^{T_i} f(\phi_t(v)) dt \simeq \frac{1}{T_i} i(\gamma_i,\mathcal{L})$$
when $T_i$ is large enough.
On the other hand, by continuity of the intersection number we have: 
$$
\lim_{T_i \rightarrow \infty} \frac{1}{T_i}i(\gamma_i, \mathcal{L})=i(\lambda, \mathcal{L})/(\pi^2(2g-2+n))= \ell(\mathcal{L})/(\pi^2(2g-2+n)).
$$
Therefore $I_{\mathcal{L},r}(t) \simeq t\ell(\mathcal{L})/(\pi^2(2g-2+n))$.
\end{proof}
\begin{rmk}
The statement of Theorem~\ref{thm:generic} can be interpreted as an extension of Crofton formula to measured laminations. Crofton formula relates the expected number of times a random line interesting a curve $\gamma$ on Euclidean plane to the length of $\gamma$ (see \cite{intgeobook}). 
\end{rmk}
Proposition~\ref{linr} and Theorem~\ref{thm:generic} imply Theorem~\ref{thm:linear}. From this result, we conclude the collection of exotic vectors has measure zero, even though it is dense by Cor.~\ref{cor:dense}.

%Recall that any measured lamination $(\mathcal{L},\mu)$ on $X$ gives a $\pi_1(X)$-invariant measured lamination $(\tilde{\mathcal{L}},\tilde\mu)$ on the universal cover $\mathbb{H}^2$. We first prove the following basic facts about rays with finite total transverse measure.

Finally, we give a proof of Theorem~\ref{sublinear}, using some elementary hyperbolic geometry. We construct $r$ in the following way. Start from a point on $\tilde{\mathcal{L}}$ and go along the leaves. At some points we take a small jump (by taking a segment transverse to $\tilde{\mathcal{L}}$) to the right of the leaf which we are on, then we land on another point of $\tilde{\mathcal{L}}$ and continue going along the leaves. If we choose the times to jump and segments properly then we can obtain a ray with the growth rate of transverse measure $\asymp f$. In preparation for a rigorous argument, we need the following lemma:
\begin{lm}\label{lm:jump}
\begin{enumerate}[label=\normalfont{(\arabic*)}, topsep=0mm, itemsep=0mm]
    \item\label{item: constants} For any $M>0$, there exists positive constants $d_1<d_2,\tau_1<\tau_2$ so that For any point $x$ on a leaf of $\tilde{\mathcal{L}}$, there exists a point $x_0$ on the same leaf with $d(x_0,x)\le M$ so that on either side of the leaf, there exists a geodesic segment $s$ satisfying:
    \begin{enumerate}[label=\normalfont{(\alph*)}, topsep=0mm, itemsep=0mm]
        \item Segment $s$ starts at $x_0$, orthogonal to the leaf and ends on another leaf;
        \item Length of $s$ is in $(d_1,d_2)$;
        \item Transverse measure of $s$ is in $(\tau_1,\tau_2)$.
    \end{enumerate}
    \item\label{item: cross_ratio} Furthermore, for each segment $s$ in \emph{(1)}, the end points of the two leaves containing $\partial s$ has cross ratio bounded away from $0$ and $\infty$.
\end{enumerate}
\end{lm}
\begin{proof}
For \ref{item: constants}, if $x$ lies on a leaf that is not the boundary of a component of $\mathbb{H}^2-\tilde{\mathcal{L}}$, then any segment starting at $x$ and perpendicular to the leaf has a positive transverse measure. Otherwise, the segment may lie completely inside a complementary region.

There are only finitely many points like $x'$ where the geodesic orthogonal to $x'$ lie completely inside a complementary region.
%For points bounded away from these finitely many points, the geodesic orthogonal to $x$ going into the region will intersects the lamination again.
Thus we may exclude some intervals on the boundary of these regions containing these finitely many points of length less than $2M$, so that the geodesic orthogonal to points outside of these intervals 
%going into the region will intersects the lamination again; as a matter of fact, it 
intersects the lamination in distance bounded above by a uniform constant.

% Now since the complementary region is lifts of finitely many components of $X-\mathcal{L}$,
We conclude that there exists $d_2>0$ so that for any point $x$ on a leaf, we can choose a point $x_0$ on the same leaf distance $\le M$ away from $x$ so that any geodesic segment starting at $x_0$, perpendicular to the leaf, and of length $d_2$ has positive transverse measure.

By compactness, orthogonal geodesic segments of length $d_2$ have positive transverse measure between $\tau_1$ and $\tau_2$.

Finally, by cutting off the end part contained completely in a complementary region, we 
%get geodesic segments for each starting point with the same transverse measure and end points on a leaf. 
assume the endpoint is on a leaf. The lengths of these modified segments are bounded below by $d_1>0$, as the transverse measure is bounded below. This concludes \ref{item: constants}.

For \ref{item: cross_ratio}, suppose otherwise. Then there exists a sequence of points $x_i$ with geodesic segments $L_i$ starting at $x_i\in\tilde{\mathcal{L}}$ and ending at $y_i\in\tilde{\mathcal{L}}$ so that the cross ratio tends to $0$ or $\infty$. Passing to a subsequence and replace $x_i$ by an element in the orbit $\pi_1(X)\cdot x_i$ if necessary, we may assume $x_i\to x_0\in\mathcal{L}$, and $y_i\to y_0\in\mathcal{L}$. It is clear the corresponding cross ratio also tends to the limit cross ratio. In particular, by assumption, the leaves which $x_0$ and $y_0$ lie on share a point at infinity. But this means that the geodesic segment $x_0y_0$ is contained in a complementary region and thus has transverse measure $0$. This contradicts the fact that the transverse measure of these segments is bounded below.
\end{proof}
For simplicity, we will call a geodesic segment described in the previous lemma a \emph{connecting segment}. We are now in a position to give a proof of Theorem~\ref{sublinear}:
\begin{proof}(of Theorem~\ref{sublinear})
Let $\{d_n\}$ be a sequence of positive numbers so that $f(d_1+\cdots+d_n)=n$. By assumption $d_n\to\infty$ as $n\to\infty$. Start with any point $p$ on a leaf of the lamination $\tilde{\mathcal{L}}$. Choose a direction and go along the leaf for a distance of $d_1$, and then turn to the left and then go along a connecting segment predicted by Lemma~\ref{lm:jump}. One may need to go forward or back along the leaf for a distance $\le M$ first, as per Lemma~\ref{lm:jump}. At the end of the segment, turn right and go along the new leaf for a distance of $d_2$, and repeat the process; See the red trajectory in Figure~\ref{fig:construction} for a schematic picture.
\begin{figure}[ht!]
    \centering
    \includegraphics{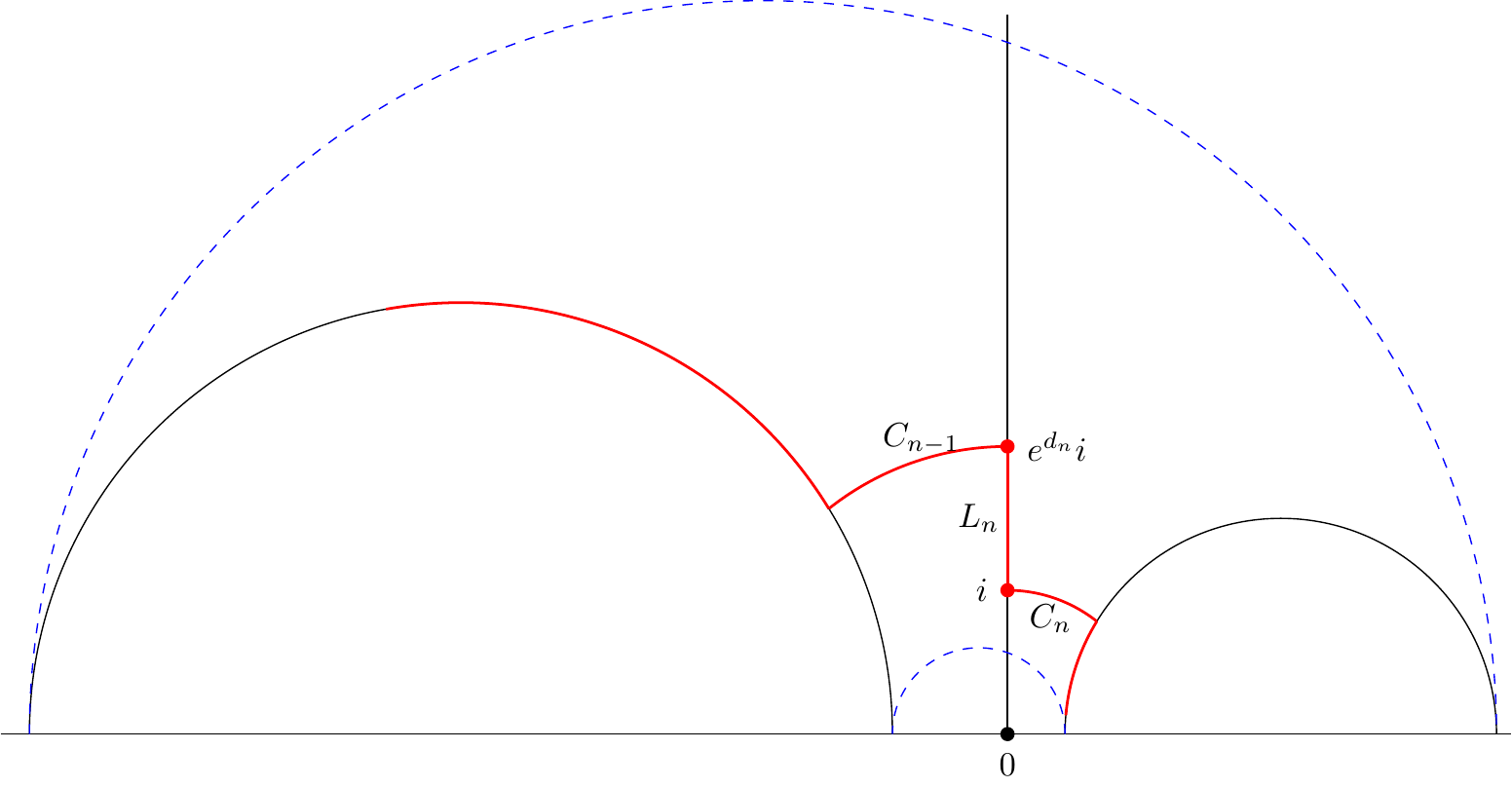}
    \iffalse
    \begin{tikzpicture}[scale=1.5]
    \draw (-7,0)--(3.5,0);
    \draw (0,0)--(0,5);
    \draw (0.4,0) arc (180:0:1.5);
    \draw (-0.8,0) arc (0:180:3);
    \draw[thick, red] (0,1) arc (90:51.607:1);
    \draw[thick, red] (0,1)--(0,2);
    \draw[thick, red] (0,2) arc (90:128.393:2);
    \draw[thick, red] (0.6210526315789474,0.7837688618520541) arc (148.499:175:1.5);
    \draw[thick, red] (-1.24211,1.56754) arc (31.5009:100:3);
    \draw[dashed, blue] (-6.8,0) arc (180:0:5.1);
    \draw[dashed, blue] (0.4,0) arc (0:180:0.6);
    
    \node[circle,fill, inner sep=1.5pt, label=below:$0$] at (0,0) {};
    \node[circle, fill, inner sep=1.5pt, label=left:$i$, red] at (0,1) {};
    \node[circle, fill, inner sep=1.5pt, label=right:$e^{d_n}i$, red] at (0,2) {};
    \node at (-0.2,1.5) {$L_n$};
    \node at (0.3,0.8) {$C_n$};
    \node at (-0.6,2.05) {$C_{n-1}$};
    \end{tikzpicture}
    \fi
    \caption{Construction of the geodesic ray}
    \label{fig:construction}
\end{figure}

It is clear that the transverse measure of this piecewise geodesic ray has the desired growth rate. It remains to show that straightening the geodesic does not change this growth rate.

We use upper half plane model to show it. Assume a leaf is the imaginary axis, and one end point of a connecting segment is at $i$, then the leaf containing the other end point is a half circle whose end points are bounded away from $0$ and $\infty$ by Part \ref{item: cross_ratio} of Lemma~\ref{lm:jump}, and also bounded away from each other as the length of the connecting segment is bounded above. %Similarly, if we put the end point of a connecting segment at $i$ and the corresponding leaf on the imaginary axis, the leaf containing the starting point is then a half circle whose end points are bounded away from $0$ and $\infty$, and are bounded apart. 
Choose constants $0<m_1<m_2$ so that the end points of the circles in either description above lies between distance $m_1$ and $m_2$ from the origin.

Now return to our construction. Denote the $n$-th segment on the leaf by $L_n$, and $n$-th connecting segment by $C_n$. Note that the length of $L_n$ is roughly $d_n$, up to a bounded constant $M$. Put $L_n$ on the imaginary axis, ending at $i$; see Figure~\ref{fig:construction}. The starting point of $L_n$ is at $e^{d_n}i$. Hence the circle on the left half plane (i.e. the leaf $L_{n-1}$ lies on) has both ends contained in the interval $[-m_2e^{d_n},-m_1e^{d_n}]$. After straightening, the geodesic will cross both circles in Figure~\ref{fig:construction}, and intersect the imaginary axis at $(0,y_n)$. The extreme cases are:
\begin{itemize}[topsep=0mm, itemsep=0mm]
\item The straightened geodesic intersects the circle on the left very far into the left end points, and the circle on the right very far into the right end points (see the larger dotted blue circle in Figure~\ref{fig:construction}), and hence $y_n\le\sqrt{e^{d_n}m_2\cdot m_2}\le e^{d_n}m_2$.
\item The straightened geodesic intersects the circle on the left very far into the right end points, and the circle on the right very far into the left end points (see the smaller dotted blue circle in Figure~\ref{fig:construction}), and hence $y_n\ge\sqrt{e^{d_n}m_1\cdot m_1}\ge m_1$.
\end{itemize}
Therefore the straightened geodesic intersects the leaf containing $L_n$ within a bounded distance ($\le\max\{|\log m_2|,|\log m_1|\}$) of $L_n$.

The leaves containing $L_n$'s and the connecting segments cut the straightened geodesic into segments. Each of these segments starts or ends on a leaf. We have the following possibilities
\begin{enumerate}[topsep=0mm, itemsep=0mm]
    \item The segment starts at a point on $L_n$ close to the end point of $L_n$. This segment then must end on $C_n$, or a point on the leaf containing $L_{n+1}$, outside $L_{n+1}$ in the backward direction along the leaf. Either way this segment has bounded length, so we can safely ignore the segment for the purpose of growth rates.
    \item The segment ends at a point on $L_n$ close to the starting point of $L_n$. Similarly, This segment then must start on $C_{n-1}$, or a point on the leaf containing $L_{n-1}$, outside $L_{n-1}$ in the forward direction along the leaf. Again this segment has bounded length.
    \item Otherwise, using triangle inequality, we conclude that this segment is $\asymp$ the portion on the piecewise geodesic ray we constructed between the same endpoints.
\end{enumerate}
From these observations we conclude that the straightened geodesic ray also has the prescribed growth rate, as desired.
\end{proof}

\appendix
\section{Appendix: Exotic rays via quadratic differentials}\label{sec:translation}
In this appendix, we give a proof of Theorem~\ref{main1} using quadratic differentials and their foliations. This approach is similar in spirit to the one adopted in Section~\ref{sec:general_proof}, but more closely resembles the discussions in the case of a punctured torus.
\subsection{Basics of quadratic differentials}
A \emph{half translation surface} is a Riemann surface $X$ with cusps together with a nonzero quadratic differential $q$ on $X$ so that $q$ has at most simple poles at the cusps of $X$. When $q=\omega^2$, where $\omega$ is a holomorphic $1$-from, the surface is called a \emph{translation surface}. Away from zeros and poles, $|q|^{1/2}$ gives a flat metric, i.e., given a path $\gamma:[0,1]\to X$ we may evaluate its length by $\int_\gamma|q|^{1/2}$. Equivalently, a half translation surface (resp. translation surface) is given by a collection of polygons in $\mathbb{C}$ with parallel sides identified in pairs via maps of the form $z\mapsto\pm z+c$ (resp. $z\mapsto z+c$). The quadratic form $dz^2$ on $\mathbb{C}$ descends to a quadratic form on the surface, and the flat metric comes from the Euclidean metric on $\mathbb{C}$. 

Given a half translation surface $(X,q)$, consider the form $\impart q^{1/2}$, which is only well defined locally, away from zeros and poles in a simply connected domain, up to the choice of a branch of the square root. Consider the foliation of $X$ by curves tangent to vectors in the kernel of $\impart q^{1/2}$, denoted by $\mathcal{F}_q$. Moreover, $|\impart q^{1/2}|$ gives a transverse measure for $\mathcal{F}_q$. In a coordinate chart so that $q=dz^2$, the leaves of $\mathcal{F}_q$ are horizontal lines, while the transverse measure is given by $|dy|$. This is the \emph{horizontal foliation} associated to $q$. Similarly, leaves of $\mathcal{F}_{-q}$ are vertical lines in the same coordinate chart, so it is the \emph{vertical foliation} of $q$. Conversely, given a pair of transverse foliations on a topological surface, there exists a complex structure and a quadratic differential realizing the pair as its horizontal and vertical foliations.

Given a complete hyperbolic surface $X$ of finite area and a measured geodesic lamination $\mathcal{L}$ on $X$, let $\mathcal{F}$ be the corresponding measured foliation. Uniformization gives a unique complex structure on $X$. By \cite[Main Theorem]{qd&foliations}, there exists a quadratic differential $q$ on $X$ so that $\mathcal{F}_q=\mathcal{F}$. So we may as well assume our geodesic lamination comes from a quadratic differential. Of course, as deforming the hyperbolic structure on $X$ has no effect on our discussion, we may simply choose a foliation transverse to $\mathcal{F}$ and construct a quadratic differential as explained above.

Given a half translation surface $(X,q)$, consider its \emph{orientation double cover} $f:\hat X\to X$ so that $f^*q=\omega^2$ for some holomorphic $1$-form on $\hat X$. In particular, $(\hat X,\omega)$ is a translation surface. The double cover $f$ is ramified at zeros of $q$ of odd order and the simple poles. Note that any pole of $q$ pulls back to a removable singularity of $\omega$. Some other structures (flat metrics, horizontal and vertical foliations) can also be lifted to $\hat X$, but not hyperbolic metrics, as near ramification points, the map cannot be an isometry. However, if we can construct an exotic ray upstairs that is bounded away from the singularities, its image under $f$, straightened in the hyperbolic metric downstairs, also gives an exotic ray. So for our purpose, we may assume the geodesic lamination comes from a holomorphic $1$-from. In particular $\mathcal{F}_q$ is orientable; by choosing an orientation for the leaves of $\mathcal{F}_q$, we may then talk about the horizontal flow without ambiguity.

By a \emph{cylinder} in $(X,q)$ we mean a maximal embedded flat cylinder. If leaves of $\mathcal{F}_q$ divide $X$ into a collection of cylinders (which we call \emph{horizontal} cylinders), the corresponding geodesic lamination is a multicurve, consisting of simple closed geodesics homotopic to the core geodesics of the cylinders. For Theorem~\ref{main1}, the assumption excludes this possibility. Let $C$ be the union of all the horizontal cylinders in $(X,q)$. Each connected component of $X\backslash C$ corresponds to a minimal component of $\mathcal{L}$. Instead of working with $(X,q)$, we focus on a connected component $X'$ of $X\backslash C$, and note that we may represent it as a collection of polygons with some edges identified in parallel pairs (these are called \emph{inner edges}), and the others forming the boundary of $X'$ (called \emph{boundary edges}). With this in mind, for simplicity, we will refer to $X'$ as $X$.

Finally, it is easy to see that we may represent $X$ as a collection of polygons in $\mathbb{C}$, so that the vertices of the polygons are precisely the zeros of $q$ and the cusps of $X$. Indeed, this is achieved by cutting the surface along flat geodesic representatives of curves starting and ending at these points.

A \emph{saddle connection} of $\mathcal{F}_q$ is a leaf starting and ending at a zero of $q$. There are at most finitely many saddle connections on $(X,q)$. As a matter of fact, we have
\begin{lm}\label{lm:saddle}
There exists a zero $x$ of $q$ so that a leaf ending at $x$ is not a saddle connection.
\end{lm}
\begin{proof}
Suppose otherwise. Since there are equal amount of leaves emanating and stopping at a zero of $q$, we must have that every leaf starting or stopping at a zero is a saddle connection. Cut the surface along these saddle connections. Each connected component we obtain is compact, supports a flat metric, and foliated by leaves homotopic to its boundary, and so is a cylinder. But we have assumed that this is not the case.
\end{proof}

\subsection{Symbolic coding and construction of exotic rays}
\paragraph{Symoblic coding.}
Recall that we view $X$ as a collection of polygons with inner edges identified in parallel pairs. Label these inner edges by letters from $E=\{e_0, \ldots, e_n, \bar{e}_0, \dots, \bar{e}_n\}$, so that each pair of parallel edges are labeled by $e_i$ and $\bar{e}_i$ for some $i$. An infinite (resp. finite) path which avoids any of the singularities gives an infinite (resp. finite) word with letters in $E$, by writing down the edges it intersects in order.

Choose an edge $e$ of the polygons that is not horizontal. Given a point $P\in e$, let $T(P)$ be the first return point on $e$ along the horizontal flow. We can define $w_T(P)$ as the word corresponding to this ray starting at $P$ and ending at $T(P)$. Similarly, we can define $w_{T^n}(p)$ and $w_{T^{\infty}}(P)$, where $T^n(P)$ is the $n$-th returning point. Note that $w_{T^n}(P)$ is not well-defined if the horizontal flow starting at $P$ hits a vertex before it intersects $e$ $n$-times. There are only finitely many points at which $w_{T^n}$ is not well-defined, and they divide $e$ into a collection $\mathcal{I}_n$ of disjoint open intervals. We have

\begin{lm} \label{split}
\begin{enumerate}[label=\normalfont{(\arabic*)}, topsep=0mm, itemsep=0mm]
    \item Each interval in $\mathcal{I}_n$ is a level set for $w_{T^n}$; that is, $w_{T^n}$ gives a constant word $w$ on an interval $I$ in $\mathcal{I}_n$, and $w_{T^n}^{-1}(w)=I$;
    \item The maximal length of the intervals in $\mathcal{I}_n$ tends to zero as $n \to \infty$.
\end{enumerate}  
\end{lm}
\begin{proof}
Note that for (1) it suffices to prove the lemma for $w_T$; the general case then follows. Let $P,Q$ be two different points on $e$. Flow the interval $PQ$ until it returns to $e$. If it does not hit any of the vertices then $w_T$ is constant on $PQ$. Otherwise, $w_T$ is different at $P$ and $Q$. This gives (1).

For (2), note that by recurrence, any interval $PQ$ as above must hit a vertex after flowing along the horizontal foliation for a sufficiently long time. This gives (2).
\end{proof}

\paragraph{Inadmissible words.}
As in the case of punctured tori, we call a word ($\mathcal{F}_q$-)inadmissible if it does not appear in the word of any half leaf. The key idea is that there is an abundance of inadmissible words; as a matter of fact, we show that there are inadmissible words corresponding to curves with arbitrarily small transverse measure (recall that in the setting of translation surfaces with horizontal foliation, the transverse measure is given by vertical distance):
\begin{prop}\label{prop:inadmissible_general}
For all $k \in \mathbb{N}$, there exists an inadmissible word $A_k$ so that $A_k$ has a representative closed curve $C_k$ of transverse measure $<c/2^k$, where $c$ is a constant independent of $k$.
\end{prop} 
\begin{proof}
By Lemma~\ref{split}, $e$ is split into intervals $I_1, \dots, I_m$ by points $P_1,\ldots,P_{m-1}$, so that each $I_j$ is a level set for $w_T$ on $e$. Note that we may choose $P=P_j$ so that it does not lie on a saddle connection; indeed, Lemma~\ref{lm:saddle} guarantees the existence of a leaf ending at a zero of $q$ but not a saddle connection. Flow along this leaf backwards until it first intersects $e$, which must be among $P_j$'s. This gives a desired point.

For simplicity, let $I=I_j$ and $I'=I_{j+1}$ be the two intervals with $P$ as one of their endpoints. Let $Q$ be a point of distance $<1/2^{k}$ from $P$ inside $I$. We construct $C_k$ as follows. Start from $Q$, flow along the leaves of $\mathcal{F}_q$ until we get to the point $T^{n}(Q)$, where $n \in \mathbb{N}$ is large enough so that
\begin{enumerate}[topsep=0mm, itemsep=0mm]
    \item Intervals in $\mathcal{I}_n$ (see Lemma~\ref{lm:saddle}) have length less than $a=|PQ|/3$;
    \item $T^n(Q)$ lies between $P$ and $Q$ and the distance between $T^n(Q)$ and $Q$ is less than $a$. 
\end{enumerate} 
Then move along $e$ and stop at a point of distance less than $1/2^k$ from $P$, inside $I'$. Flow along the horizontal flow and stop at a point of distance less than $1/2^k$ from $Q$. Move horizontally to $Q$. This gives a closed curve $C_k$.

The transverse measure $C_k$ is less than 
$$c\left(a+\frac{1}{2^{k}}+\frac{1}{2^{k}}\right) < \frac{3c}{2^{k}},$$
where $c$ depends on the slope of $e$. This closed curve corresponds to a word $A_k$. In order to prove $A_k$ is inadmissible, we show that any admissible word starting with $w_{T^n}(Q)$ is followed by $w_T(Q)=w_T(I)$ afterwards. Indeed, note that the interval $w^{-1}_{T^n}(w_{T^n}(Q))$ has length less than $a$, so its image under $T^n$ lies inside $I$. On the other hand, in $A_k$ after $w_{T^n}(Q)$ the word $w_T(I')$ follows. Therefore $A_k$ is inadmissible.   
\end{proof}

\paragraph{Constructing exotic rays.}
We are in a position to prove Theorem~\ref{main1}, and the construction itself is similar to that in the case of punctured tori. However, there are some issues arise, see Lemma~\ref{lm:close} and Proposition~\ref{prop:tail}.

\begin{proof}(of Theorem ~\ref{main1}) Proposition~\ref{prop:inadmissible_general} gives for each $k$ an inadmissible word $A_k$ represented by a curve $C_k$ based a point $Q=Q_k$ of distance $<1/2^k$ from $P$. Consider the concatenation $w=A_1A_2\cdots$. This infinite word is represented by a piecewise ray starting at $Q_1$. The ray goes along $C_1$ until it gets back to $Q_1$ and then moves along $e$ to $Q_2$ (the distance moved is $<1/2$), and then continues along $C_2$ and so on.

The transverse measure of this ray is $\leq$ transverse measures of $C_k$'s + transverse measures of the moves along $e$, which is
$$< c\sum _{i=1}^{\infty} \frac{1}{2^k}+ c'\sum_{i=1}^{\infty} \frac{1}{2^k}<\infty.$$
Furthermore, since any tail of $w$ contains an inadmissible word, Lemma~\ref{lm:close} and Proposition~\ref{prop:tail} imply that the ray is not asymptotic to any leaf. To get uncountably many, a similar argument to that in \S\ref{sec:general_proof} suffices.
\end{proof}
In the case of punctured tori, the unit square with four vertices removed gives a fundamental domain whose boundary is geodesic both in the flat metric and the hyperbolic metric. Moreover, straightening an infinite ray with respect to the hyperbolic metric does not change its word, as the corresponding homotopy does not pass through the vertices, as the cusps are infinite distance away. These are no longer true for the general case, so we need careful analysis of our construction.

First, note that after straightening, each zero of $q$ blows up to an ideal quadrilateral. Choose an point in each quadrilateral, and connect these points with hyperbolic geodesic segments, so that these segments are homotopic rel end points to the flat geodesics we cut up to obtain the polygonal representation of $(X,q)$. Cutting the surface along these hyperbolic geodesic segments, we obtain a fundamental domain for the hyperbolic surface; hyperbolic elements corresponding to gluing pairs of sides of this fundamental domain give essentially the same coding system as polygons in flat metric. Moreover, the word for a leaf of $\mathcal{F}_q$ with respect to the horizontal flow in the polygonal representation is exactly the same as the word for the straightened leaf with respect to this new fundamental domain. We will call these chosen points and the cusps of $X$ \emph{singularities}.

In the proof of Theorem~\ref{main1} above, the constructed ray consists of segments along the horizontal flow and along $e$. We may thus reconstruct this ray in the hyperbolic setting by going accordingly along straightened leaves and along the geodesic segment corresponding to $e$. We have:
\begin{lm}\label{lm:close}
\begin{enumerate}[topsep=0mm, itemsep=0mm, label=\normalfont{(\arabic*)}]
    \item The straightened ray is eventually in the $\epsilon$-neighborhood of the geodesic lamination $\mathcal{L}$ for any $\epsilon>0$;
    \item Straightening the ray in the proof above does not change the tail of its word.
\end{enumerate}
\end{lm}
\begin{proof}
We first prove (1) assuming $\mathcal{L}$ is maximal, and $X$ does not have any boundary component (recall that we restrict our attention to a subsurface containing $\mathcal{L}$, so $X$ may have boundary components). In this case, each complementary region is an ideal triangle, and the complement of $\epsilon$-neighborhood of the geodesic lamination consists of a finite number of domains like the gray region in Figure~\ref{fig:epsilon_neighborhood}.
\begin{figure}[ht!]
\centering
\begin{minipage}[b]{0.48\textwidth}
    \centering
    \includegraphics{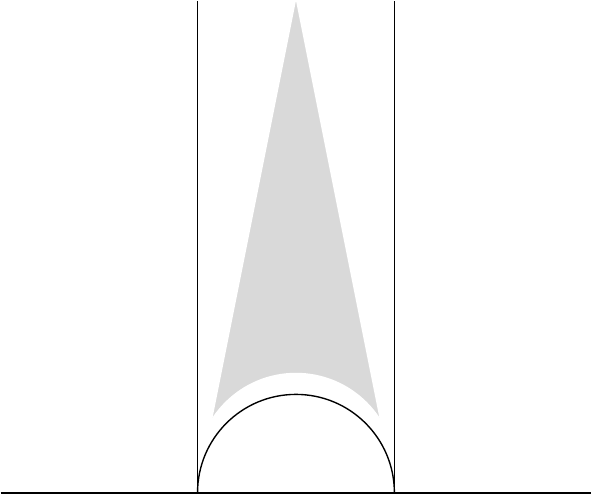}
    \iffalse
    \begin{tikzpicture}
    \begin{scope}
    \clip (-1,0)--(1,0)--(1,5)--(-1,5)--cycle;
    \fill[gray, opacity=0.3] (-1,0)--(1,0)--(1,5)--(-1,5)--cycle;
    \fill[white] (-1,0)--(0,5)--(-2,5.1)--(-2,0)--cycle;
    \fill[white] (1,0)--(0,5)--(2,5.1)--(2,0)--cycle;
    \fill[white] (0,0.2) circle (1.01980390272);
    \end{scope}
    \draw (-3,0)--(3,0);
    \draw (-1,0)--(-1,5);
    \draw (1,0)--(1,5);
    \draw (1,0) arc (0:180:1);
    \end{tikzpicture}
    \fi
    \caption{Complement of $\epsilon$-neighborhood of $\mathcal{L}$}
    \label{fig:epsilon_neighborhood}
\end{minipage}
\begin{minipage}[b]{0.48\textwidth}
    \centering
    \includegraphics{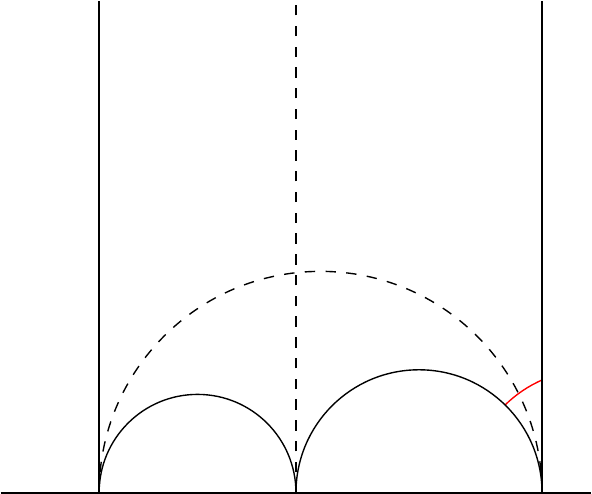}
    \iffalse
    \begin{tikzpicture}
    \begin{scope}
    \clip (-1,0)--(3.5,0)--(3.5,5)--(-1,5)--cycle;
    \draw[red] (4,0) circle (1.25);
    \fill[white] (0,0) circle (1);
    \fill[white] (2.25,0) circle (1.25);
    \end{scope}
    \draw (-2,0)--(4,0);
    \draw (-1,0)--(-1,5);
    \draw[dashed] (1,0)--(1,5);
    \draw (1,0) arc (0:180:1);
    \draw (1,0) arc (180:0:1.25);
    \draw[dashed] (-1,0) arc (180:0:2.25);
    \draw (3.5,0)--(3.5,5);
    \end{tikzpicture}
    \fi
    \caption{A quadrilateral complementary region}
    \label{fig:quadrilateral}
\end{minipage}
\end{figure}

If a geodesic ray passes through the gray area, it then intersects two of edges of the ideal triangle. Moreover, at least one of the nonobtuse angles formed by the ray and the edges is bounded below by a positive constant depending only on $\epsilon$. A compactness argument as in the proof of Proposition~\ref{prop:leaf} then implies that the segment of length $L$ outside the ideal triangle on the geodesic ray with one end at the vertex of that angle has transverse measure bounded below by a positive constant depending only on $\epsilon$ and $L$. Since the ray we constructed has finite transverse measure, it must eventually avoid the gray region.

In general, a component of $X\backslash\mathcal{L}$ may be an ideal polygon with more than $3$ edges. The argument above only implies that the geodesic ray eventually either stays in an $\epsilon$-neighborhood of $\mathcal{L}$ or get closer to a diagonal of the polygon (see Figure~\ref{fig:quadrilateral}; a geodesic ray very close to a diagonal there intersects the edges in angles that can be arbitrarily small). To exclude such a possibility, we look at the construction in the proof of Theorem~\ref{main1}.

Recall that the ray constructed in Theorem~\ref{main1} consists of (usually very long) segments along the leaves and (relatively short) jumps along $e$. From the flat picture to the hyperbolic picture, each point where $w_{T^\infty}$ or $w_{T^{-\infty}}$ is undefined ``blows up" to a segment cross a complementary region. For simplicity, we will call such a point \emph{singular}. The length of these segments $\to 0$ for any infinite sequence of them. In particular, they must eventually be a segment across a tip of a ideal polygon (e.g. the red segment in Figure~\ref{fig:quadrilateral}). Moreover, recall that the fixed point $P$ is chosen so that it does not lie on a saddle connection, so it also blows up to a segment across a tip (indeed, the backward flow at $P$ does not hit any singularities). Thus by discarding a starting portion, we may assume that all singular points the ray crosses on $e$ blow up to segments across a tip.

Straightening the ray does not change how it intersects each complementary region. For each tip crossing segment, the two edges forming the tip have two end points not shared by both of them, and the geodesic determined by those two end points is the only diagonal the crossing segment could possibly be close to after straightening. But this is not possible, as one of the end points of the straightened geodesic lies bounded distance away from the two end points mentioned above. Indeed, Figure~\ref{fig:bounded_away} illustrates this: going from left to right in the , we flow along a segment $L$ on a leaf then jump along the edge $e$, and then go along a leaf again. Note that the length of $L$ tends to infinity, and each jump along $e$ has length bounded above. So the starting point on the left of the picture is bounded away from the complementary region $R$. In particular, even after straightening, the geodesic ray cannot be close to the diagnal of $R$.
\begin{figure}[ht!]
\centering
\includegraphics{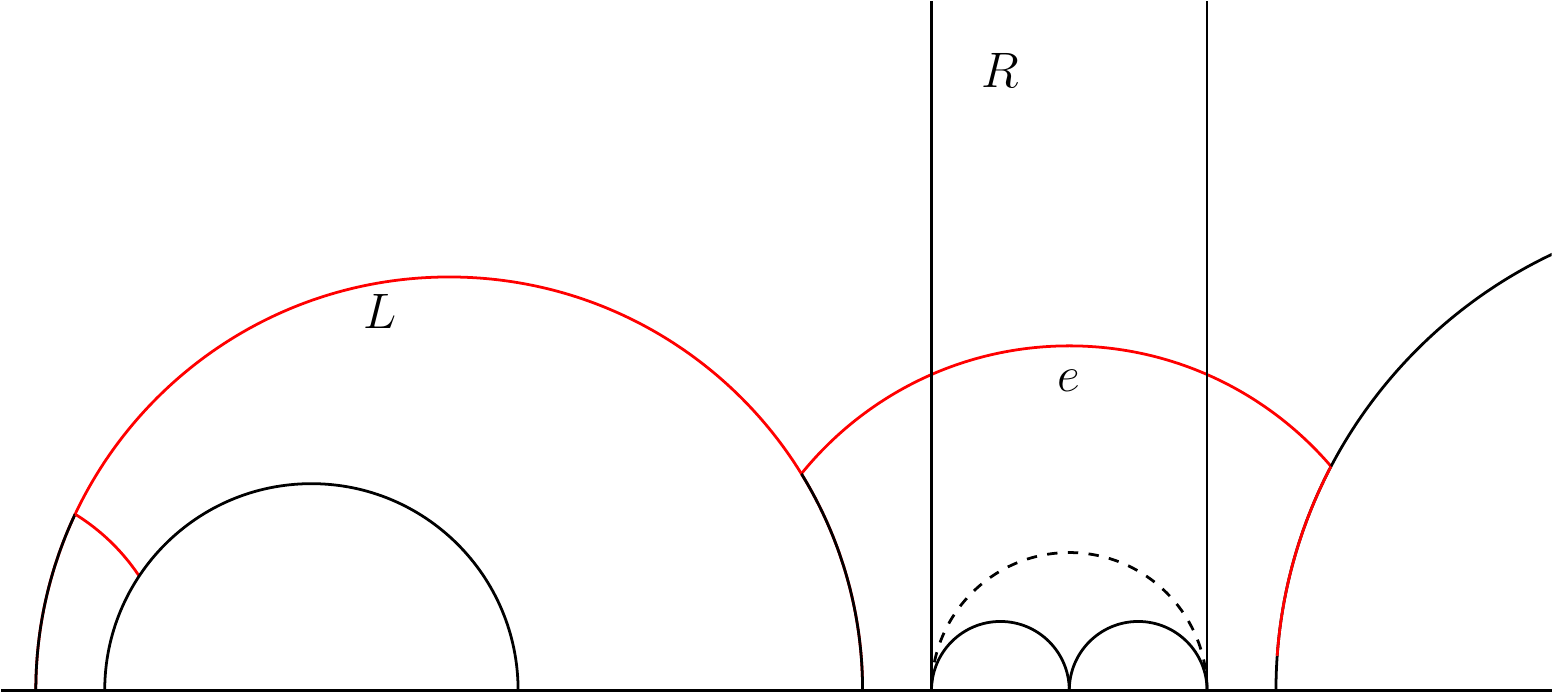}
\iffalse
\begin{tikzpicture}[scale=0.7, thick]
\begin{scope}
\clip (-13.5,0)--(9,0)--(9,10)--(-13.5,10);
\draw[red] (2,0) circle (5);
\draw[fill=white] (12,0) circle (7);
\draw[red, fill=white] (-7,0) circle (6);
\begin{scope}
\clip (2,0) circle (5);
\draw (-7,0) circle (6);
\end{scope}
\begin{scope}
\clip (-14,0) circle (3);
\draw (-7,0) circle (6);
\end{scope}
\begin{scope}
\clip (-7,0) circle (6);
\draw[red] (-14,0) circle (3);
\draw [fill=white] (-9,0) circle (3);
\end{scope}
\begin{scope}
\clip (2,0) circle (5);
\draw [red] (12,0) circle (7);
\fill [white] (5,0) circle (0.5);
\end{scope}
\begin{scope}
\clip (5,0) circle (0.5);
\draw (12,0) circle (7);
\end{scope}
\end{scope}
\draw (-13.5,0)--(9,0);
\draw (0,0)--(0,10);
\draw (4,0)--(4,10);
\draw (0,0) arc (180:0:1);
\draw (2,0) arc (180:0:1);
\draw[dashed] (4,0) arc (0:180:2);

\node at (2,4.5) {\Large $e$};
\node at (-8,5.5) {\Large $L$};
\node at (1,9) {\Large $R$};
\end{tikzpicture}
\fi
\caption{After straightening, the ray is still far way from the dashed diagonal}
\label{fig:bounded_away}
\end{figure}
The argument is similar when $X$ has a boundary component. This gives (1).

For (2), note that any crossing stays in an $\epsilon$-neighborhood of $\mathcal{L}$ after straightening, so it doesn't cross any singularities.
\end{proof}

\begin{prop}\label{prop:tail}
Suppose two geodesic rays are bounded away from any singularities. If they are asymptotic then the corresponding words have the same tail.
\end{prop}
\begin{proof}
Indeed, if two geodesic rays are asymptotic, some lifts of them share an end point. Since they are bounded away from singularities, there are at most finitely many lifts of singularities between the two lifts. They must thus share the same word after certain point.
\end{proof}
\begin{rmk}\label{rmk:same}
The approach described in this section produces very much the same family of exotic rays as in Section~\ref{sec:general_proof}. Indeed, since we showed that the exotic rays constructed in this section eventually fall into $\epsilon$-neighborhoods of the lamination, they are carried by the corresponding train tracks. The approach with translation surfaces is constructive and easily programmable, as in the punctured tori case, but we need to show that the straightened geodesic ray is bounded away from the singularities. For the approach with train tracks, this is baked into the construction: a train track carrying the geodesic lamination $\mathcal{L}$ can be homotoped into a neighborhood of $\mathcal{L}$. However, the inadmissible words are chosen in a nonconstructive way.
\end{rmk}

\bibliographystyle{math}
\bibliography{biblio}
\end{document}